\ifdefined \ptrf
\documentclass[smallextended]{svjour3}       
\smartqed
\else
\documentclass[preprint,10pt]{imsart}
\usepackage[hyperindex,breaklinks]{hyperref}
\fi
\usepackage{graphicx}

\ifdefined \ptrf
\else \usepackage{amsthm}
\fi
\usepackage{amsmath}
\usepackage{amssymb}

\usepackage{wasysym}
\usepackage{bbm}
\usepackage{bm}
\usepackage{mathrsfs}
\usepackage{mathtools}
\usepackage{dsfont}

\usepackage{color}

\usepackage{hyperref}

\usepackage{xspace}

\newcommand{\Holder}{H\"{o}lder\xspace}

\newcommand{\metricspace}{\mathcal{Z}}

\newtheorem{thm}{Theorem}
\newtheorem{cor}[thm]{Corollary}
\ifdefined\ptrf
\else
\newtheorem{lemma}[thm]{Lemma}
\newtheorem{conjecture}[thm]{Conjecture}
\fi
\newtheorem{prop}[thm]{Proposition}

\newtheorem{fact}[thm]{Fact}

\newtheorem{defn}[thm]{Definition}

\newtheorem{rem}[thm]{Remark}

\numberwithin{equation}{section}

\numberwithin{thm}{section}

\newcommand{\cororef}[1]{Corollary~{\ref{cor:#1}}}

\newcommand{\lemref}[1]{Lemma~{\ref{lem:#1}}}
\newcommand{\defref}[1]{Definition~{\ref{def:#1}}}
\newcommand{\secref}[1]{Section~{\ref{sec:#1}}}

\DeclareMathAlphabet{\mathsfsl}{OT1}{cmss}{m}{sl}

\newcommand{\qtext}[1]{\quad\text{#1}\quad}

\newcommand{\term}{\emph}

\newcommand{\notate}[1]{\textcolor{red}{\textbf{[#1]}}}

\renewcommand{\phi}{\varphi}

\newcommand{\II}{\mathbbm{1}}

\newcommand{\defby}{\mathrel{\mathop:}=}

\newcommand{\half}{\tfrac{1}{2}}

\newcommand{\econst}{\mathrm{e}}
\newcommand{\iunit}{\mathrm{i}}

\newcommand{\Id}{\mathbf{I}}

\newcommand{\zeromtx}{\bm{0}}

\newcommand{\coll}[1]{\mathscr{#1}}

\newcommand{\R}{\mathbb{R}}
\newcommand{\C}{\mathbb{C}}

\newcommand{\M}{\mathbb{M}}
\newcommand{\Sym}[1]{\mathbb{H}^{#1}}

\newcommand{\abs}[1]{\left\vert {#1} \right\vert}
\newcommand{\abssq}[1]{{\abs{#1}}^2}

\newcommand{\real}{\operatorname{Re}}
\newcommand{\imag}{\operatorname{Im}}

\newcommand{\diff}[1]{\mathrm{d}{#1}}
\newcommand{\idiff}[1]{\, \diff{#1}}
\newcommand{\ddt}[1]{\frac{\mathrm{d}}{\mathrm{d}{#1}}}

\newcommand{\Prob}[1]{\mathbb{P}\left\{ {#1} \right\}}
\newcommand{\Expect}{\operatorname{\mathbb{E}}}
\newcommand{\Var}{\operatorname{Var}}

\newcommand{\condl}{\, \vert \,}
\newcommand{\bcondl}{\, \big\vert \,}

\newcommand{\vct}[1]{\bm{#1}}
\newcommand{\mtx}[1]{\bm{#1}}

\newcommand{\adj}{*}

\newcommand{\trace}{\operatorname{tr}}
\newcommand{\ntr}{\operatorname{\bar{tr}}}

\def\indep{\perp\!\!\!\perp} 

\newcommand{\psdle}{\preccurlyeq}
\newcommand{\psdge}{\succcurlyeq}

\newcommand{\ip}[2]{\left\langle {#1},\ {#2} \right\rangle}
\newcommand{\absip}[2]{\abs{\ip{#1}{#2}}}

\newcommand{\norm}[1]{\left\Vert {#1} \right\Vert}
\newcommand{\normsq}[1]{\norm{#1}^2}

\newcommand{\smnorm}[2]{{\bigl\Vert {#2} \bigr\Vert}_{#1}}

\newcommand{\fnorm}[1]{\norm{#1}_{\mathrm{F}}}
\newcommand{\fnormsq}[1]{\fnorm{#1}^2}

\newcommand{\indnorm}[2]{\norm{#2}_{#1\to#1}} 
\newcommand{\pnorm}[2]{\norm{#2}_{S_{#1}}} 

\newcommand{\Zvec}{Z}

\newcommand{\tv}[0]{d_{\text{TV}}}

\newcommand{\OurAbstract}{
This paper establishes new concentration inequalities
for random matrices constructed from independent random variables.
These results are analogous with the generalized Efron--Stein inequalities
developed by Boucheron et al.  The proofs rely on
the method of exchangeable pairs.}

\ifdefined\ptrf

\usepackage[numbers]{natbib}

\begin{document}

\title{Efron--Stein Inequalities for Random Matrices
}

\titlerunning{Matrix Efron--Stein Inequalities}        

\author{
Daniel~Paulin\and \nolinebreak Lester~Mackey \and \nolinebreak Joel~A.~Tropp
}

\authorrunning{D.~Paulin\and  L.~Mackey \and J.~A.~Tropp} 

\institute{Daniel~Paulin \at
             Department of Mathematics, National University of Singapore. \\
             \email{paulindani@gmail.com}          
             \and
             Lester~Mackey \at
             Department of Statistics, Stanford University. \\
             \email{lmackey@stanford.edu}
             \and
             Joel~A.~Tropp \at
             Department of Computing and Mathematical Sciences, California Institute of Technology.\\  	   \email{jtropp@cms.caltech.edu}
}

\date{Date: 11 May 2014.  Received: date / Accepted: date.}

\maketitle

\begin{abstract}
\OurAbstract
\end{abstract}

\else

\usepackage{natbib}

\begin{document}

\begin{frontmatter}
\title{Efron--Stein Inequalities \\ for Random Matrices}
\runtitle{Matrix Efron--Stein Inequalities}

\begin{aug}
\author{\fnms{Daniel} \snm{Paulin}\ead[label=e1]{paulindani@gmail.com}}, 
\author{\fnms{Lester} \snm{Mackey}\ead[label=e2]{lmackey@stanford.edu}}
\and
\author{\fnms{Joel A.} \snm{Tropp}\thanksref{tJ}\ead[label=e3]{jtropp@cms.caltech.edu}}
\runauthor{Paulin, Mackey, and Tropp}
\thankstext{tJ}{Tropp was supported by ONR awards N00014-08-1-0883 and N00014-11-1002, AFOSR award FA9550-09-1-0643, and a Sloan Research Fellowship.}

\affiliation{National University of Singapore, Stanford University, \and California Institute of Technology}

\address{D.~Paulin\\
Department of Statistics and Applied Probability\\
National University of Singapore\\
6 Science Drive 2, Block S16, 06-127\\
Singapore 117546\\
\printead{e1}}

\address{L.~Mackey\\
Department of Statistics\\
Stanford University\\
Sequoia Hall\\
390 Serra Mall\\
Stanford, California 94305-4065\\
USA\\
\printead{e2}}

\address{J.~A.~Tropp\\
Department of Computing and Mathematical Sciences\\
California Institute of Technology\\
Annenberg Center, Room 307\\
1200 E. California Blvd.\\
Pasadena, California 91125\\
\printead{e3}}
\end{aug}

\begin{abstract}
\OurAbstract
\end{abstract}


\end{frontmatter}
\fi

\vspace{1pc}

\section{Introduction} \label{sec:intro}

Matrix concentration inequalities provide probabilistic bounds for
the spectral-norm deviation of a random matrix from its mean value.
The monograph~\citep{Tro14:User-Friendly-FnTML} contains an overview
of this theory and an extensive bibliography.
This machinery has revolutionized the analysis of
non-classical random matrices that arise in
statistics~\citep{koltchinskii2012neumann},
machine learning~\citep{MorvantKoRa12},
signal processing~\citep{NJS13:Phase-Retrieval},
numerical analysis~\citep{AT14:Effective-Stiffness},
theoretical computer science~\citep{WX08:Derandomizing-Ahlswede-Winter},
and combinatorics~\citep{Oli11:Spectrum-Random}.

In the scalar setting, the core concentration results concern sums
of independent random variables.  Likewise, in the matrix setting,
the central results concern independent sums.  For example, the
matrix Bernstein inequality~\cite[Thm.~1.4]{Tro11:User-Friendly-FOCM}
describes the behavior of independent, centered random matrices that
are subject to a uniform bound.  There are also a few results that apply
to more general classes of random matrices, e.g., the matrix bounded
difference inequality~\cite[Cor.~7.5]{Tro11:User-Friendly-FOCM} and 
the dependent matrix inequalities of~\cite{MackeyJoChFaTr12}.
Nevertheless, it is common to encounter random matrices that we cannot
treat using these techniques.

In the scalar setting, there are concentration inequalities that can provide
information about the fluctuations of more complicated random variables.
In particular, Efron--Stein inequalities~\citep{BoLuMa2003,BoLuMa2005}
describe the concentration of functions of independent random variables 
in terms of random estimates for the local Lipschitz behavior of those functions.
These results have found extensive applications~\citep{BoLuMa2013}.

The goal of this paper is to establish new Efron--Stein
inequalities that describe the concentration properties
of a matrix-valued function of independent random variables.
The main results appear below as Theorems~\ref{thm:mxmoment}
and~\ref{thm:mxEfronStein}.

To highlight the value of this work, we establish
an improved version of the matrix bounded difference inequality
(Corollary~\ref{cor:bound-diff}).
We also develop a more substantial application to compound
sample covariance matrices (Theorem~\ref{prop:compound}).

We anticipate that our results have many additional consequences.
For instance, we envision new proofs of consistency
for correlation matrix estimation~\cite{QimanWenXin,TonyCaiJiangTiefeng}
and inverse covariance matrix estimation~\cite{Ravikumar} under sparsity constraints.

\begin{rem}[Prior Work]
This paper significantly extends and updates our earlier
report~\cite{PMT13:Deriving-Matrix}.
In particular, the matrix Efron--Stein inequalities are new.
The application to compound sample covariance matrices is also new.
The manuscript~\cite{PMT13:Deriving-Matrix} will not be published.
\end{rem}

\subsection{Technical Approach}

In the scalar setting, the generalized Efron--Stein inequalities
were originally established using entropy methods~\cite{BoLuMa2003,BoLuMa2005}.
Unfortunately, in the matrix setting, entropy methods do not
seem to have the same strength~\citep{CT14:Subadditivity-Matrix}.

Instead, our argument is based on ideas from the method of exchangeable
pairs~\citep{Stein72,Stein86}.  In the scalar setting, this approach for proving 
concentration inequalities was initiated in the paper~\citep{Cha07:Steins-Method}
and the thesis~\citep{Cha08:Concentration-Inequalities}.
The extension to random matrices appears in the recent paper~\citep{MackeyJoChFaTr12}.

The method of exchangeable pairs has two chief advantages over alternative
approaches to matrix concentration.
First, it offers a straightforward way to prove polynomial moment inequalities,
which are not easy to obtain using earlier techniques.
Second, exchangeable pair arguments also apply to random matrices
constructed from weakly dependent random variables.

The paper~\citep{MackeyJoChFaTr12} focuses on sums of weakly
dependent random matrices because the techniques are less effective
for general matrix-valued functionals.  In this work,
we address this shortcoming by developing a matrix version
of the kernel coupling construction from~\cite[Sec.~4.1]{Cha08:Concentration-Inequalities}.
This argument requires some challenging new matrix inequalities
that may have independent interest.  We also describe some new
techniques for controlling the evolution of the kernel coupling.

We believe that our proof of the Efron--Stein inequality via
the method of exchangeable pairs is novel, even in the scalar setting.
As a consequence, our paper contributes to the growing literature
that uses Stein's ideas to develop concentration inequalities.

\section{Notation and Preliminaries from Matrix Analysis}
\label{sec:notation}

This section summarizes our notation,
as well as some background results from matrix analysis.
The reader may prefer to skip this material at first;
we have included detailed cross-references throughout the paper.

\subsection{Elementary Matrices}

First, we introduce the identity matrix $\Id$ and the zero matrix $\zeromtx$.
The standard basis matrix $\mathbf{E}_{ij}$ has a one in the $(i, j)$ position
and zeros elsewhere.  The dimensions of these matrices are determined by context.

\subsection{Sets of Matrices and the Semidefinite Order}

We write $\M^d$ for the algebra of $d \times d$ complex matrices.  The \term{trace}
and \term{normalized trace} are given by
$$
\trace \mtx{B} = \sum_{i=1}^d b_{ii}
\quad\text{and}\quad
\ntr \mtx{B} = \frac{1}{d} \sum_{i=1}^d b_{ii}
\quad\text{for $\mtx{B} \in \M^d$.}
$$
The symbol $\norm{ \cdot }$ always refers to the usual operator norm on $\M^d$ induced by the $\ell_2^d$ vector norm.  We also equip $\M^d$ with the trace inner product $\ip{ \mtx{B} }{ \mtx{C} } \defby \trace[\mtx{B}^* \mtx{C}]$ to form a Hilbert space.  

Let $\Sym{d}$ denote the real-linear subspace of $\M^d$ consisting of $d \times d$ Hermitian matrices.
The cone of positive-semidefinite matrices will be abbreviated as $\Sym{d}_+$.
Given an interval $I$ of the real line, we also define $\Sym{d}(I)$
to be the convex set of Hermitian matrices whose eigenvalues are all contained in $I$.

We use curly inequalities, such as $\psdle$, for the positive-semidefinite order
on the Hilbert space $\Sym{d}$. That is, for $\mtx{A}, \mtx{B} \in \Sym{d}$, we write
$\mtx{A} \psdle \mtx{B}$ if and only if $\mtx{B} - \mtx{A}$ is positive semidefinite.

\subsection{Matrix Functions}

Let $f : I \to \R$ be a function on an interval $I$ of the real line.  We can lift $f$ to form a \term{standard matrix function} $f : \Sym{d}(I) \to \Sym{d}$.  More precisely, for each matrix $\mtx{A} \in \Sym{d}(I)$, we define the standard matrix function via the rule
$$
f(\mtx{A}) := \sum\nolimits_{k=1}^d f(\lambda_k) \, \vct{u}_k \vct{u}_k^\adj,
\quad\text{where}\quad
\mtx{A} = \sum\nolimits_{k=1}^d \lambda_k \, \vct{u}_k \vct{u}_k^\adj
$$
is an eigenvalue decomposition of the Hermitian matrix $\mtx{A}$.  When we apply a familiar scalar function to an Hermitian matrix, we are always referring to the associated standard matrix function.  To denote general matrix-valued functions, we use bold uppercase letters, such as $\mtx{F}, \mtx{H}, \mtx{\Psi}$.

\subsection{Monotonicity \& Convexity of Trace Functions}

The trace of a standard matrix function inherits certain properties from the scalar function.
Let $I$ be an interval, and assume that $\mtx{A}, \mtx{B} \in \Sym{}(I)$.
When the function $f : I \to \R$ is weakly increasing,
\begin{equation} \label{eqn:trace-monotone}
\mtx{A} \psdle \mtx{B}
\quad\text{implies}\quad
\trace f(\mtx{A}) \leq \trace f(\mtx{B}).
\end{equation}
When the function $f : I \to \R$ is convex,
\begin{equation} \label{eqn:trace-convex}
\trace f( \tau \mtx{A} + (1 - \tau) \mtx{B}) \leq \tau \trace f(\mtx{A}) + (1-\tau) \trace f(\mtx{B})
\quad\text{for $\tau \in [0,1]$.}
\end{equation}
See~\cite[Props.~1 and 2]{Petz94} for proofs.

\subsection{The Real Part of a Matrix and the Matrix Square}

For each matrix $\mtx{M}\in \M^d$, we introduce the real
and imaginary parts,
\begin{equation} \label{eqn:reim}
\begin{aligned}
\real(\mtx{M}) &:= \tfrac{1}{2} (\mtx{M}+\mtx{M}^*) \in \Sym{d}
\quad\text{and}\quad \\
\imag(\mtx{M}) &:= \tfrac{1}{2\iunit} (\mtx{M} - \mtx{M}^\adj) \in \Sym{d}.
\end{aligned}
\end{equation}
Note the semidefinite bound
\begin{equation} \label{eqn:re-square}
\real(\mtx{M})^2 \psdle \tfrac{1}{2} (\mtx{MM}^\adj + \mtx{M}^\adj \mtx{M})
\quad\text{for each $\mtx{M} \in \M^d$.}
\end{equation}
Indeed, $\real(\mtx{M})^2 + \imag(\mtx{M})^2 = \half (\mtx{MM}^\adj + \mtx{M}^\adj \mtx{M})$
and $\imag(\mtx{M})^2 \psdge \mtx{0}$.

The real part of a product of Hermitian matrices satisfies
\begin{equation} \label{eqn:matrix-am-gm}
\real(\mtx{AB}) = \frac{\mtx{AB} + \mtx{BA}}{2}
\psdle \frac{\mtx{A}^2 + \mtx{B}^2}{2}
\quad\text{for all $\mtx{A}, \mtx{B} \in \Sym{d}$.}
\end{equation}
This result follows when we expand $(\mtx{A} - \mtx{B})^2 \psdge \mtx{0}$.  As a consequence,
\begin{equation} \label{eqn:square-convex}
\left(\frac{\mtx{A} + \mtx{B}}{2}\right)^2 \psdle \frac{\mtx{A}^2 + \mtx{B}^2}{2}
\quad\text{for all $\mtx{A}, \mtx{B} \in \Sym{d}$.}
\end{equation}
In other words, the matrix square is operator convex.

\subsection{Some Matrix Norms}

Finally, we will make use of two additional families of matrix norms.  For $p \in [1, \infty]$, the Schatten $p$-norm is given by
\begin{equation} \label{eqn:schatten-norm}
\pnorm{p}{\mtx{B}} \defby \big( \trace \abs{\mtx{B}}^p \big)^{1/p}
	\quad\text{for each $\mtx{B} \in \M^d$} ,
\end{equation}
where $\abs{\mtx{B}} \defby (\mtx{B}^\adj \mtx{B})^{1/2}$.  For $p \geq 1$, we introduce the matrix norm induced by the $\ell_p^d$ vector norm:
\begin{equation} \label{eqn:induced-norm}
\indnorm{p}{\mtx{B}} \defby \sup_{\vct{x} \neq \vct{0}} \ \frac{\norm{\mtx{B}\vct{x}}_p}{\norm{\vct{x}}_p}
\quad\text{for each $\mtx{B} \in \M^d$}
\end{equation}
In particular, the matrix norm induced by the $\ell_1^d$ vector norm returns the maximum $\ell_1^d$ norm of a column; the norm induced by $\ell_\infty^d$ returns the maximum $\ell_1^d$ norm of a row.

\section{Matrix Moments and Concentration}
\label{sec:moment-concentration}

Our goal is to develop expectation and tail bounds
for the spectral norm of a random matrix.  As in the scalar
setting, these results follow from bounds for polynomial and exponential
moments.  This section describes the mechanism by which we convert
bounds for matrix moments into concentration inequalities.

\subsection{The Matrix Chebyshev Inequality} \label{sec:matrix-chebyshev}

We can obtain concentration inequalities
for a random matrix in terms of the Schatten $p$-norm.  This
fact extends Chebyshev's inequality.

\begin{prop}[Matrix Chebyshev Inequality] \label{prop:chebyshev}
Let $\mtx{X} \in \Sym{d}$ be a random matrix.  For all $t > 0$,
$$
\Prob{ \norm{\mtx{X}} \geq t } \leq \inf_{p \geq 1} t^{-p} \cdot \Expect \pnorm{p}{\mtx{X}}^p.
$$
Furthermore,
$$
\Expect \norm{\mtx{X}} \leq \inf_{p \geq 1} \big( \Expect \pnorm{p}{\mtx{X}}^p \big)^{1/p}.
$$
\end{prop}

\noindent
This statement repeats~\cite[Prop.~6.2]{MackeyJoChFaTr12}.
See also~\cite[App.]{AW02:Strong-Converse} for earlier work.

\subsection{The Matrix Laplace Transform Method} \label{sec:matrix-laplace}

We can also obtain exponential concentration inequalities from a matrix
version of the moment generating function.

\begin{defn}[Trace Mgf]
Let $\mtx{X}$ be a random Hermitian matrix.  The
\term{(normalized) trace moment generating function} of $\mtx{X}$ is defined as
$$
m(\theta) \defby m_{\mtx{X}}(\theta) \defby \Expect \ntr \econst^{\theta \mtx{X}}
\quad\text{for $\theta \in \R$.}
$$
\end{defn}

\noindent
We believe this definition is due to~\cite{AW02:Strong-Converse}.

The following proposition is an extension of Bernstein's method.
It converts bounds for the trace mgf
of a random matrix into bounds on its maximum eigenvalue.

\begin{prop} [Matrix Laplace Transform Method] \label{prop:matrix-laplace}
Let $\mtx{X} \in \Sym{d}$ be a random matrix with normalized trace mgf
$m(\theta) \defby \Expect \ntr \econst^{\theta \mtx{X}}$.  For each $t \in \R$, 
\begin{align} 
\Prob{\lambda_{\max} ( \mtx{X}) \geq t}
	&\leq d \cdot \inf_{\theta > 0} \ \exp\{ -\theta t + \log m(\theta) \},
	\label{eqn:laplace-upper-tail} \\
\Prob{\lambda_{\min}(\mtx{X}) \leq t}
	&\leq d \cdot \inf_{\theta < 0} \ \exp\{ -\theta t + \log m(\theta) \}.
	\label{eqn:laplace-lower-tail}
\end{align}
Furthermore,
\begin{align}
\Expect \lambda_{\max}(\mtx{X})
	\leq \inf_{\theta > 0} \ \frac{1}{\theta}\, [\log d + \log m(\theta)],  \label{eqn:laplace-upper-mean} \\
\Expect \lambda_{\min}(\mtx{X})
	\geq \sup_{\theta < 0} \ \frac{1}{\theta}\, [\log d + \log m(\theta)].  \label{eqn:laplace-lower-mean}
\end{align}
\end{prop}

\noindent
Proposition~\ref{prop:matrix-laplace} restates~\cite[Prop.~3.3]{MackeyJoChFaTr12},
which collects results
from~\cite{AW02:Strong-Converse,Oli10:Sums-Random,Tro11:User-Friendly-FOCM,CGT11:Masked-Sample}.

We will use a special case of Proposition~\ref{prop:matrix-laplace}.
This result delineates the consequences of a specific bound for the trace mgf.

\begin{prop}\label{prop:gaussexp}
Let $\mtx{X} \in \Sym{d}$ be a random matrix with normalized trace mgf
$m(\theta) \defby \Expect \ntr \econst^{\theta \mtx{X}}$.  Assume that
there are nonnegative constants $c, v$ for which
$$
\log m(\theta) \leq \frac{v \theta^2}{2(1 - c \theta)}
\quad\text{when $0 \leq \theta < 1/c$.}
$$
Then, for all $t \geq 0$,
\begin{equation}
\Prob{ \lambda_{\max}(\mtx{X}) \geq t}\leq d \exp\left(\frac{-t^2}{2v+2ct}\right).
\end{equation}
Furthermore,
$$
\Expect \lambda_{\max}(\mtx{X}) \leq \sqrt{2 v \log d} + c \log d.
$$
\end{prop}

\noindent
See \cite[Sec.~4.2.4]{MackeyJoChFaTr12} for the proof of Proposition~\ref{prop:gaussexp}.

\section{Matrix Efron--Stein Inequalities}
\label{sec:main}

The main outcome of this paper is a family of Efron--Stein inequalities for random matrices.
These estimates provide powerful tools for controlling the trace moments of a random matrix
in terms of the trace moments of a randomized ``variance proxy.''
Combining these inequalities with the results from Section~\ref{sec:moment-concentration},
we can obtain concentration inequalities for the spectral norm.

\subsection{Setup for Efron--Stein Inequalities}
\label{sec:setup}

Efron--Stein inequalities apply to random matrices constructed from a family
of independent random variables.  Introduce the random vector
$$
Z := (Z_1, \dots, Z_n) \in \metricspace
$$
where $Z_1, \dots, Z_n$ are mutually independent random variables.
We assume that $\metricspace$ is a Polish space to avoid problems
with conditioning~\cite[Thm.~12.2.2]{Dud02:Real-Analysis}.
Let $\mtx{H} : \metricspace \to \Sym{d}$ be a measurable function that takes values in the
space of Hermitian matrices, and construct the centered random matrix
$$
\mtx{X} := \mtx{X}(Z) := \mtx{H}(Z) - \Expect \mtx{H}(Z).
$$
Our goal is to study the behavior of $\mtx{X}$, which describes the fluctuations
of the random matrix $\mtx{H}(Z)$ about its mean value.
We will assume that $\Expect \normsq{\mtx{X}} < \infty$
so that we can discuss variances.

A function of independent random variables will concentrate about its mean
if it depends smoothly on all of its inputs.  We can quantify smoothness by
assessing the influence of each coordinate on the matrix-valued function.
For each coordinate $j$, construct the random vector
$$
Z^{(j)} := (Z_1, \dots, Z_{j-1}, \widetilde{Z}_j, Z_{j+1}, \dots, Z_n) \in \metricspace
$$
where $\widetilde{Z}_j$ is an independent copy of $Z_j$.
It is clear that $Z$ and $Z^{(j)}$ have the same distribution,
and they differ only in coordinate $j$.  Form the random matrices
\begin{equation} \label{eqn:Xj}
\mtx{X}^{(j)} := \mtx{X}(Z^{(j)}) = \mtx{H}(Z^{(j)}) - \Expect \mtx{H}(Z)
\quad\text{for $j = 1, \dots, n$.}
\end{equation}
Note that each $\mtx{X}^{(j)}$ follows the same distribution as $\mtx{X}$.

Efron--Stein inequalities control the fluctuations of the centered random matrix $\mtx{X}$
in terms of the discrepancies between $\mtx{X}$ and the $\mtx{X}^{(j)}$.
To present these results, let us define the \term{variance proxy}
\begin{equation} \label{eq:Vdef}
\mtx{V} := 
	\frac{1}{2} \sum\nolimits_{j=1}^n \Expect\big[ \big( \mtx{X} - \mtx{X}^{(j)} \big)^2 \bcondl Z \big].
\end{equation}
Efron--Stein inequalities bound the trace moments of the random
matrix $\mtx{X}$ in terms of the moments of the variance proxy $\mtx{V}$.
This is similar to the estimate provided by a Poincar{\'e} inequality~\cite[Sec.~3.5]{BoLuMa2013}.

Passing from the random matrix $\mtx{X}$ to the variance proxy $\mtx{V}$ has a number of advantages.
There are many situations where the variance proxy admits an accurate deterministic bound,
so we can reduce problems involving random matrices to simpler matrix arithmetic.
Moreover, the variance proxy is a sum of positive semidefinite terms,
which are easier to control than arbitrary random matrices.
The examples in Sections~\ref{sec:self-bounding}, \ref{sec:matrix-bdd-diff},
and~\ref{sec:compound-sample-covar} support these claims.

\begin{rem}
In the scalar setting, Efron--Stein inequalities~\citep{BoLuMa2003,BoLuMa2005}
can alternatively be expressed in terms of the positive part of the fluctuations:
$$
V_+ := \frac{1}{2} \sum\nolimits_{j=1}^n \Expect\big[ \big(X - X^{(j)} \big)_+^2 \bcondl Z \big]
$$
where $(a)_+ := \max\{0, a\}$.  Our approach can reproduce these 
positive-part bounds in the scalar setting but does not deliver
positive-part expressions in the general matrix setting.  
See Section~\ref{sec:conjectures} for more discussion.
\end{rem}

\subsection{Polynomial Efron--Stein Inequalities for Random Matrices}

The first main result of the paper is a polynomial Efron--Stein inequality for a random matrix
constructed from independent random variables.

\begin{thm}[Matrix Polynomial Efron--Stein]\label{thm:mxmoment}
Instate the notation of Section~\ref{sec:setup}, and assume that $\Expect \normsq{\mtx{X}} < \infty$.
For each natural number $p \geq 1$,
$$
\big( \Expect \pnorm{2p}{\mtx{X}}^{2p} \big)^{1/(2p)} 
	\leq \sqrt{2(2p-1)} \left( \Expect \pnorm{p}{\mtx{V}}^p \right)^{1/(2p)}.
$$
\end{thm}

\noindent
The proof appears in Section~\ref{sec:poly-moment}.

We can regard Theorem~\ref{thm:mxmoment} as a matrix extension of the scalar
concentration inequality~\cite[Thm.~1]{BoLuMa2005}, which was obtained using the entropy method.
In contrast, our results depend on a different style of argument, based on the theory of exchangeable
pairs~\citep{Stein86,Cha08:Concentration-Inequalities}.  Our approach is novel, even in the scalar setting.
Unfortunately, it leads to slightly worse constants.

Theorem~\ref{thm:mxmoment} allows us to control the trace moments of a random Hermitian matrix
in terms of the trace moments of the variance proxy.  We can obtain
probability inequalities for the spectral norm by combining this result
with the matrix Chebyshev inequality, Proposition~\ref{prop:chebyshev}.

\subsection{Exponential Efron--Stein Inequalities for Random Matrices} \label{sec:efronstein}

The second main result of the paper is an exponential Efron--Stein inequality for
a random matrix built from independent random variables.

\begin{thm}[Matrix Exponential Efron--Stein]\label{thm:mxEfronStein}
Instate the notation of Section~\ref{sec:setup},
and assume that $\norm{\mtx{X}}$ is bounded.
When $\abs{\theta} \leq \sqrt{\psi/2}$,
\begin{equation}\label{eq:momgenVeq}
\log \Expect \ntr \econst^{\theta \mtx{X}}
	\leq \frac{\theta^2/\psi}{1- 2\theta^2/\psi}
	\log \Expect \ntr \econst^{\psi\mtx{V}}.
\end{equation}
\end{thm}

\noindent
The proof appears in Section~\ref{sec:exp-es-proof}.

Theorem~\ref{thm:mxEfronStein} is a matrix extension of the exponential
Efron--Stein inequalities for scalar random variables established in~\cite[Thm.~1]{BoLuMa2003}
by means of the entropy method.  As in the polynomial case, we use a new argument based on
exchangeable pairs. 

Theorem~\ref{thm:mxEfronStein} allows us to control trace exponential moments
of a random Hermitian matrix in terms of the trace exponential moments of the
variance proxy.  We arrive at probability inequalities for the spectral norm
by combining this result with the matrix Laplace transform method,
Proposition~\ref{prop:matrix-laplace}.  Although bounds on polynomial
trace moments are stronger than bounds on exponential
trace moments~\cite[Sec.~6]{MackeyJoChFaTr12},
the exponential inequalities are often more useful in practice.

\begin{rem}[Weaker Integrability Conditions]
Theorem~\ref{thm:mxEfronStein} holds under weaker
regularity conditions on $\mtx{X}$, but we have
chosen to present the result here to avoid
distracting technical arguments.
\end{rem}

\subsection{Rectangular Matrices}

Suppose now that $\mtx{H} : \metricspace \to \C^{d_1 + d_2}$ is a measurable
function taking \emph{rectangular} matrix values.  We can also develop Efron--Stein
inequalities for the random rectangular matrix $\mtx{X} := \mtx{H}(Z) - \Expect \mtx{H}(Z)$
as a formal consequence of the results for Hermitian random matrices.

The approach is based on a device from operator theory called the
\term{Hermitian dilation}, which is defined as
$$
\coll{H}(\mtx{B}) := \begin{bmatrix} \mtx{0} & \mtx{B} \\ \mtx{B}^\adj & \mtx{0} \end{bmatrix}
	\in \Sym{d_1 + d_2}
	\quad\text{for $\mtx{B} \in \C^{d_1 + d_2}$.}
$$
To obtain Efron--Stein inequalities for random rectangular matrices,
we simply apply Theorem~\ref{thm:mxmoment}
and Theorem~\ref{thm:mxEfronStein} to the dilation $\coll{H}(\mtx{X})$.
We omit the details.  For more information about these arguments,
see \cite[Sec.~2.6]{Tro11:User-Friendly-FOCM}, \cite[Sec.~8]{MackeyJoChFaTr12}, or
\cite[Sec.~2.1.13]{Tro14:User-Friendly-FnTML}.

\section{Example: Self-Bounded Random Matrices}
\label{sec:self-bounding}

As a first example, we consider the case where the variance proxy
is dominated by an affine function of the centered random matrix.

\begin{cor}[Self-Bounded Random Matrices]
\label{cor:self-bounded}
Instate the notation of Section~\ref{sec:setup}.
Assume that $\norm{\mtx{X}}$ is bounded, and suppose that
there are nonnegative constants $c, v$ for which
\begin{equation} \label{eqn:self-bound}
\mtx{V} \psdle v \Id + c \mtx{X}
\quad\text{almost surely.}
\end{equation}
Then, for all $t \geq 0$,
$$
\Prob{ \lambda_{\max}(\mtx{X}) \geq t }
	\leq d \exp\left( \frac{-t^2}{4v + 6ct} \right).
$$
Furthermore,
$$
\Expect \lambda_{\max}(\mtx{X}) \leq \sqrt{4v \log d} + 3 c \log d.
$$
\end{cor}

\begin{proof}
The result is an easy consequence of the exponential Efron--Stein inequality
for random matrices, Theorem~\ref{thm:mxEfronStein}.  When $0 \leq \theta < 1/(3c)$,
we may calculate that
\begin{align}
\log m_{\mtx{X}}(\theta) = \log \Expect \ntr \econst^{\theta \mtx{X}}
	&\leq \frac{\theta^2 / \psi}{1 - 2\theta^2/\psi}
	\log \Expect \ntr \econst^{\psi \mtx{V}} \notag \\
	&\leq \frac{\theta^2 / \psi}{1 - 2\theta^2/\psi}
	\log \Expect \ntr \econst^{\psi (v \Id + c \mtx{X})} \notag \\
	&= \frac{\theta^2 / \psi}{1 - 2\theta^2/\psi} \left( \psi v + \log \Expect \ntr \econst^{\psi c \mtx{X}} \right).
	\label{eqn:self-bdd-pf}
\end{align}
In the first inequality, we can introduce the bound~\eqref{eqn:self-bound} for $\mtx{V}$ 
because the trace exponential is monotone~\eqref{eqn:trace-monotone}.
Select $\psi = \theta/c$ to obtain a copy of
$m_{\mtx{X}}(\theta)$ on the right-hand side of~\eqref{eqn:self-bdd-pf}.
Solve for $m_{\mtx{X}}(\theta)$ to reach
\begin{equation} \label{eqn:self-bdd-pf-2}
\log m_{\mtx{X}}(\theta)
	\leq \frac{v \theta^2}{1 - 3c \theta}
	\quad\text{when $0 \leq \theta < 1/(3c)$.}
\end{equation}
Invoke Proposition~\ref{prop:gaussexp} to complete the proof.
\end{proof}

The hypothesis~\eqref{eqn:self-bound} is analogous with the assumptions
in the result~\cite[Thm.~4.1]{MackeyJoChFaTr12}.  In Section~\ref{sec:matrix-bdd-diff},
we explain how this estimate supports a matrix version of the bounded difference inequality.
But the result also extends well beyond this example.

\section{Example: Matrix Bounded Differences}
\label{sec:matrix-bdd-diff}

The matrix bounded difference inequality~\citep{Tro11:User-Friendly-FOCM,MackeyJoChFaTr12}
controls the fluctuations of a matrix-valued function of independent
random variables.
This result has been used to analyze algorithms for
multiclass classification~\citep{MachartRa12,MorvantKoRa12},
crowdsourcing~\citep{DalviDaKuRa13}, and non-differentiable optimization~\citep{ZhouHu13}.

Let us explain how to derive a refined version of the matrix bounded differences inequality
from Theorem~\ref{thm:mxEfronStein}.

\begin{cor}[Matrix Bounded Differences] \label{cor:bound-diff}
Instate the notation of Section~\ref{sec:setup}.
Assume there are deterministic matrices $\mtx{A}_1, \dots, \mtx{A}_n \in \Sym{d}$
for which
\begin{equation} \label{eqn:bdd-diff-hyp}
\big(\mtx{H}(z_1, \dots, z_n) - \mtx{H}(z_1, \dots, z_j', \dots, z_n)\big)^2
	\psdle \mtx{A}_j^2
	\quad\text{for each index $j$.}
\end{equation}
In this expression, $z_k$ and $z_k'$ range over all possible value of $Z_k$.
Compute the boundedness parameter
\begin{equation} \label{eqn:bdd-diff-sigma2}
\sigma^2 \defby \norm{ \sum\nolimits_{j=1}^n \mtx{A}_j^2 }.
\end{equation}
Then, for all $t\geq 0$, 
$$
\Prob{ \lambda_{\max}\left( \mtx{H}(Z) - \Expect \mtx{H}(Z) \right) \geq t }
	\leq d \cdot \econst^{-t^2/(2\sigma^2)}.
$$
Furthermore,
$$
\Expect \lambda_{\max}\left( \mtx{H}(Z) - \Expect \mtx{H}(Z) \right) \leq \sigma \sqrt{2\log d}.
$$ 
\end{cor}

\begin{proof}
Observe that the variance proxy satisfies
$$
\begin{aligned}
\mtx{V} &= \frac{1}{2} \sum_{j=1}^n \Expect \big[ \big(\mtx{X} - \mtx{X}^{(j)} \big)^2 \bcondl Z \big] \\
	&= \frac{1}{2} \sum_{j=1}^n \Expect\big[ \big(\mtx{H}(Z) - \mtx{H}(Z^{(j)}) \big)^2 \bcondl Z \big]
	\psdle \frac{1}{2} \sum_{j=1}^n \mtx{A}_j^2.
\end{aligned}
$$
It follows from the definition~\eqref{eqn:bdd-diff-sigma2} that $\mtx{V} \psdle \half \sigma^2 \, \Id$.
Invoke Corollary~\ref{cor:self-bounded} to complete the argument.
\end{proof}

\begin{rem}[Related Work]
Corollary~\ref{cor:bound-diff} improves the constants in~\cite[Cor.~7.5]{Tro11:User-Friendly-FOCM},
and it removes an extraneous assumption from~\cite[Cor.~11.1]{MackeyJoChFaTr12}.
It is possible to further improve the constants in the exponent by a factor of 2 to obtain a bound of the form $d \cdot \econst^{-t^2/\sigma^2}$; see the original argument in \citep{PMT13:Deriving-Matrix}.
\end{rem}

\begin{rem}[Weaker Hypotheses]
We can relax the assumption~\eqref{eqn:bdd-diff-hyp} to read
\begin{align*}
\sum_{j=1}^n \big(\mtx{H}(z_1, \dots, z_n) - \mtx{H}(z_1, \dots, z_j', \dots, z_n)\big)^2
	&\psdle \mtx{A}^2
\end{align*}
where $\mtx{A} \in \Sym{d}$ is a fixed Hermitian matrix.  In this case, $\sigma^2 := \norm{\mtx{A}}^2$.
\end{rem}

\section{Application: Compound Sample Covariance Matrices}
\label{sec:compound-sample-covar}

In this section, we consider the \emph{compound sample covariance matrix}:
\begin{equation}
\label{eqn:compound-covariance}
\widehat{\mtx{\Lambda}}_n := \frac{1}{n}\mtx{Z} \mtx{B}\mtx{Z}^\adj.
\end{equation}
The central matrix $\mtx{B} \in \Sym{n}$ is fixed, and
the columns of $\mtx{Z} \in \C^{p \times n}$ are random vectors
drawn independently from a common distribution on $\C^p$.

When the matrix $\mtx{B} = n^{-1} \Id$, the compound sample covariance matrix $\widehat{\mtx{\Lambda}}_n$
reduces to the classical empirical covariance $n^{-1} \mtx{ZZ}^\adj$.  The
latter matrix can be written as a sum of independent rank-one matrices,
and its concentration properties are well established~\citep{AdamczakComptes}.
For general $\mtx{B}$, however, the random matrix $\widehat{\mtx{\Lambda}}_n$
cannot be expressed as an independent sum,
so the behavior becomes significantly harder to characterize.
See, for example, the analysis of~\cite{Sol14:Error-Bound}.

The most common example of a compound sample covariance matrix
is the compound Wishart matrix~\citep{Speicher}, where
the columns of $\mtx{Z}$ are drawn from a multivariate normal distribution.
These matrices have been used to estimate the sample covariance
under correlated sampling~\citep{burda2011applying}.  They also
arise in risk estimation for portfolio management~\citep{collins2013compound}.

We will use Theorem~\ref{thm:mxEfronStein} to develop an exponential
concentration inequality for one class of compound sample covariance matrices.

\begin{thm}[Concentration of Compound Sample Covariance]\label{prop:compound}
Suppose that the entries of $\mtx{Z}\in \C^{p\times n}$ are iid random variables
with mean zero, variance $\sigma^2$, and magnitude bounded by $L$.
Let $\mtx{B} \in \Sym{n}$ be fixed.  For any $t\geq 0$ we have
\begin{multline*}
\Prob{ \norm{ \mtx{Z}\mtx{B}\mtx{Z}^\adj - \Expect[\mtx{Z}\mtx{B}\mtx{Z}^\adj] } \geq t } \\
	\leq 2p \exp\left( \frac{-t^2}{44(p \sigma^2 + L^2) \fnormsq{\mtx{B}} + 32\sqrt{3} Lp \norm{\mtx{B}} t} \right).
\end{multline*}
Furthermore,
$$
\Expect \norm{ \mtx{Z}\mtx{B}\mtx{Z}^\adj - \Expect[\mtx{Z}\mtx{B}\mtx{Z}^\adj] }
\le 
2\sqrt{44(p \sigma^2 + L^2)\log p  \fnormsq{\mtx{B}}} + 32\sqrt{3} Lp \log p \norm{\mtx{B}}.
$$ 
\end{thm}

It is possible to obtain finer results when $\mtx{B}$ is positive semidefinite.
We have also made a number of loose estimates in order to obtain a clear
statement of the bound.

\subsection{Setup}

Let $\mtx{Z}$ be a $p \times n$ random matrix whose entries
are independent, identically distributed, zero-mean random variables
with variance $\sigma^2$ and bounded in magnitude by $L=1$.
The general case follows by a homogeneity argument.
Define the centered random matrix
$$
\mtx{X}(\mtx{Z}) = \mtx{ZBZ}^\adj - \Expect[ \mtx{ZBZ}^\adj ]
$$
where $\mtx{B} \in \Sym{d}$. By direct calculation, the expectation takes the form
\begin{equation} \label{eqn:expect-zbz}
\Expect[ \mtx{ZBZ}^\adj ] = \sigma^2 (\trace \mtx{B}) \, \Id.
\end{equation}
As in \secref{setup}, we introduce independent copies $\widetilde{Z}_{ij}$
of the entries $Z_{ij}$ of $\mtx{Z}$ and define the random matrices
$$
\mtx{Z}^{(ij)} = \mtx{Z} + (\widetilde{Z}_{ij} - Z_{ij}) \, \mathbf{E}_{ij}
\quad\text{for $i = 1, \dots, p$ and $j = 1, \dots, n$.}
$$
Introduce the variance proxy
\begin{equation} \label{eqn:wish-V}
\mtx{V} 
	:= \frac{1}{2}\sum_{i=1}^p \sum_{j=1}^n  \Expect \left[ \big(\mtx{X}(\mtx{Z}) - \mtx{X}(\mtx{Z}^{(ij)}) \big)^2
	\bcondl \mtx{Z} \right].
\end{equation}
Theorem~\ref{thm:mxEfronStein} allows us to bound $\lambda_{\max}(\mtx{X}(\mtx{Z}))$ in terms of 
the trace mgf of $\mtx{V}$.  Our task is to develop bounds on the trace mgf of $\mtx{V}$ in terms
of the problem data.

\subsection{A Bound for the Variance Proxy}

We begin with a general bound for $\mtx{V}$.  First, we use the definitions
to simplify the expression~\eqref{eqn:wish-V},
and then we invoke the operator convexity~\eqref{eqn:square-convex} of the square function:
$$
\begin{aligned}
\mtx{V} &= \frac{1}{2} \sum\nolimits_{ij} \Expect \left[ \big( 2 (\widetilde{Z}_{ij} - Z_{ij}) \real( \mathbf{E}_{ij} \mtx{BZ}^\adj )
	+ \abssq{\smash{\widetilde{Z}_{ij} - Z_{ij}}} \, \mathbf{E}_{ij} \mtx{B} \mathbf{E}_{ij}^\adj \big)^2 \bcondl \mtx{Z} \right] \\
	&\psdle \frac{1}{2} \sum\nolimits_{ij} \Expect \left[  8 \abssq{\smash{\widetilde{Z}_{ij} - Z_{ij}}} \real( \mathbf{E}_{ij} \mtx{BZ}^\adj)^2
	+ 2 \abs{\smash{\widetilde{Z}_{ij} - Z_{ij}}}^4 \, \abssq{\smash{b_{jj}}} \, \mathbf{E}_{ii} \bcondl \mtx{Z} \right],
\end{aligned}
$$
where $\{b_{ij}\}_{1\le i,j\le n}$ denote the elements of $\mtx{B}$.
Since $Z_{ij}$ and $\widetilde{Z}_{ij}$ are centered variables that are bounded in magnitude by one,
$$
\Expect \big[ \abssq{\smash{\widetilde{Z}_{ij} - Z_{ij}}} \condl \mtx{Z} \big] \leq 2
\quad\text{and}\quad
\Expect \big[ \abs{\smash{\widetilde{Z}_{ij} - Z_{ij}}}^4 \condl \mtx{Z} \big] \leq 8.
$$
Using the bound~\eqref{eqn:re-square} for the square of the real part, we obtain
\begin{align} \label{eqn:V-bound-gen}
\mtx{V}	&\psdle \sum\nolimits_{ij} \big[ 4
	(\mtx{BZ}^\adj \mtx{ZB})_{jj} \, \mathbf{E}_{ii}  + 4 \mtx{ZB} \mathbf{E}_{jj} \mtx{BZ}^\adj
	+ 8 \, \abssq{\smash{b_{jj}}} \, \mathbf{E}_{ii} \big] \notag \\
	&= 4p \ntr[ \mtx{ZB}^2 \mtx{Z}^\adj ] \, \Id + 4p \, \mtx{Z}\mtx{B}^2\mtx{Z}^\adj + 8 \left(\sum\nolimits_j \abssq{\smash{b_{jj}}} \right) \, \Id.
\end{align}
In the first term on the right-hand side of~\eqref{eqn:V-bound-gen}, we have used cyclicity of the standard trace, and then we have rescaled to obtain the normalized trace.

\subsection{A Bound for the Trace Mgf of the Random Matrix}

Next, we apply the matrix exponential Efron--Stein inequality, Theorem~\ref{thm:mxEfronStein},
to bound the logarithm of the trace mgf of the random matrix.
\begin{equation} \label{eqn:wish-ES}
\log \Expect \ntr \econst^{\theta \mtx{X}}
	\leq \frac{\theta^2/\psi}{1 - 2\theta^2/\psi}
	\log \Expect \ntr \econst^{\psi \mtx{V}}.
\end{equation}
Let us focus on the trace mgf of $\mtx{V}$.
The trace exponential is monotone~\eqref{eqn:trace-monotone},
so we can introduce the bound~\eqref{eqn:V-bound-gen} for $\mtx{V}$ and simplify
the expression:
$$
\begin{aligned}
\log \Expect \ntr \econst^{\psi \mtx{V}}
	&\leq \log \Expect \left[ \econst^{4\psi p \ntr[ \mtx{ZB}^2 \mtx{Z}^\adj]}
	\ntr \econst^{4\psi p \, \mtx{ZB}^2\mtx{Z}^\adj} \right]
	+ 8 \psi \left(\sum\nolimits_j \abssq{\smash{b_{jj}}} \right) \\
	&\leq \frac{1}{2} \log \Expect \econst^{8\psi p \, \ntr[ \mtx{ZB}^2 \mtx{Z}^\adj]}
	+ \frac{1}{2} \log \Expect \ntr \econst^{8 \psi p \, \mtx{ZB}^2\mtx{Z}^\adj}
	+ 8 \psi \left(\sum\nolimits_j \abssq{\smash{b_{jj}}} \right) \\
	&\leq \log \Expect\ntr \econst^{8 \psi p \, \mtx{ZB}^2 \mtx{Z}^\adj}
	+ 8 \psi \left(\sum\nolimits_j \abssq{\smash{b_{jj}}} \right).
\end{aligned}
$$
To reach the second line, we use the Cauchy--Schwarz inequality for expectation, and
we use Jensen's inequality to pull the normalized trace through the square.  To arrive
at the last expression, we apply Jensen's inequality to draw out the normalized
trace from the exponential.
Substitute the last display into~\eqref{eqn:wish-ES} and write out
the definition of $\mtx{X}$ to conclude that
\begin{multline} \label{eqn:wish-mgf-gen}
\log \Expect \ntr \econst^{\theta(\mtx{ZBZ}^\adj - \Expect[\mtx{ZBZ}^\adj])} \\
	\leq \frac{\theta^2/\psi}{1 - 2\theta^2/\psi} \left[
	\log \Expect \ntr \econst^{8\psi p \, \mtx{ZB}^2 \mtx{Z}^\adj}
	+ 8\psi \left( \sum\nolimits_j \abssq{\smash{b_{jj}}} \right) \right].
\end{multline}
This mgf bound~\eqref{eqn:wish-mgf-gen} is the central point in the argument.
The rest of the proof consists of elementary (but messy) manipulations.

\subsection{The Positive-Semidefinite Case}

First, we develop an mgf bound for a compound sample covariance matrix based
on a positive-semidefinite matrix $\mtx{A} \psdge \mtx{0}$.
Invoke the bound~\eqref{eqn:wish-mgf-gen} with the choice $\mtx{B} = \mtx{A}$,
and introduce the estimate $\mtx{A}^2 \psdle \norm{\mtx{A}} \mtx{A}$ to reach
\begin{multline*}
\log \Expect \ntr \econst^{\theta(\mtx{ZAZ}^\adj - \Expect[\mtx{ZAZ}^\adj])} \\
	\leq \frac{\theta^2/\psi}{1 - 2\theta^2/\psi} \left[
	\log \Expect \ntr \econst^{8\psi p \norm{\mtx{A}}\, \mtx{ZAZ}^\adj}
	+ 8 \psi (\max\nolimits_j a_{jj} )(\trace \mtx{A}) \right].
\end{multline*}
Select $\psi = \theta / (8p\norm{\mtx{A}})$, which yields
\begin{multline*}
\log \Expect \ntr \econst^{\theta(\mtx{ZAZ}^\adj - \Expect[\mtx{ZAZ}^\adj])} \\
	\leq \frac{1}{1 - 16p \norm{\mtx{A}} \theta} \left[ 8p \norm{\mtx{A}} \theta
	\log \Expect \ntr \econst^{\theta \, \mtx{ZAZ}^\adj}
	+ 8 \theta^2 (\max\nolimits_j a_{jj} )(\trace \mtx{A}) \right].
\end{multline*}
Referring to the calculation~\eqref{eqn:expect-zbz}, we see that
$$
\log \Expect \ntr \econst^{\theta\, \mtx{ZAZ}^\adj}
	= \log \Expect \ntr \econst^{\theta (\mtx{ZAZ}^\adj - \Expect[\mtx{ZAZ}^\adj])}
	+ \sigma^2 (\trace \mtx{A}) \, \theta.
$$
Combine the last two displays, and rearrange to arrive at
\begin{equation} \label{eqn:wish-psd-mgf}
\log \Expect \ntr \econst^{\theta(\mtx{ZAZ}^\adj - \Expect[\mtx{ZAZ}^\adj])}
	\leq \frac{8\theta^2 \trace \mtx{A}}{1 - 24p \norm{\mtx{A}}\theta}
	\left( p\sigma^2 \norm{\mtx{A}}
	+ \max\nolimits_j a_{jj} \right).
\end{equation}
At this point, we can derive probabilistic bounds for
$\lambda_{\max}( \mtx{ZAZ}^\adj - \Expect[ \mtx{ZAZ}^\adj ] )$
by applying Corollary~\ref{cor:self-bounded}.

\subsection{The General Case}

To analyze the case where $\mtx{B} \in \Sym{n}$
is arbitrary, we begin once again with~\eqref{eqn:wish-mgf-gen}.
To control the mgf on the right-hand side,
we need to center the random matrix $\mtx{ZB}^2\mtx{Z}^\adj$.
Applying the calculation~\eqref{eqn:expect-zbz} with $\mtx{B} \mapsto \mtx{B}^2$,
we obtain
$$
\log \Expect \ntr \econst^{8\psi p \, \mtx{ZB}^2 \mtx{Z}^\adj}
	= \log \Expect \ntr \econst^{8\psi p \, (\mtx{ZB}^2 \mtx{Z}^\adj - \Expect[\mtx{ZB}^2\mtx{Z}^\adj])}
	+ 8 p \sigma^2 \fnormsq{\mtx{B}} \psi.
$$
We have used the fact that $\trace \mtx{B}^2 = \fnormsq{\mtx{B}}$.
Since $\mtx{B}^2$ is positive semidefinite, we may introduce the
bound~\eqref{eqn:wish-psd-mgf} with $\mtx{A}= \mtx{B}^2$ and $\theta= 8 \psi p$.
This step yields
$$
\log \Expect \ntr \econst^{8\psi p \, \mtx{ZB}^2 \mtx{Z}^\adj}
	\leq \frac{512 p^2 \fnormsq{\mtx{B}} \normsq{\mtx{B}} ( p \sigma^2  + 1 ) \psi^2 }{1 - 192 p^2 \normsq{\mtx{B}}\psi}
	+ 8 p \sigma^2 \fnormsq{\mtx{B}} \psi.
$$
This argument relies on the estimate $\max_{j} (\mtx{B}^2)_{jj} \leq \normsq{\mtx{B}}$.

Introduce the latter display into~\eqref{eqn:wish-mgf-gen}.
Select $\psi = (384 p^2 \normsq{\mtx{B}})^{-1}$,
and invoke the inequality $\sum_{j} \abssq{b_{jj}} \leq \fnormsq{\mtx{B}}$.
A numerical simplification delivers
\begin{align*}
\log \Expect \ntr \econst^{\theta(\mtx{ZBZ}^\adj - \Expect[\mtx{ZBZ}^\adj])}
	&\leq \frac{11\fnormsq{\mtx{B}}( p \sigma^2  + 1 )\theta^2 }{1-768p^2\normsq{\mtx{B}} \theta^2}\\
	&= \frac{11\fnormsq{\mtx{B}}( p \sigma^2  + 1 )\theta^2 }{(1-\sqrt{768}p\norm{\mtx{B}} \theta)(1+\sqrt{768}p\norm{\mtx{B}} \theta)} \\
	&\leq \frac{11\fnormsq{\mtx{B}}( p \sigma^2  + 1 )\theta^2 }{1-\sqrt{768}p\norm{\mtx{B}} \theta}.
\end{align*}
Tail and expectation bounds for the maximal eigenvalue follow from Proposition~\ref{prop:gaussexp}
with $v=22\fnormsq{\mtx{B}}(p\sigma^2+1)$ and $c=16\sqrt{3}p\norm{\mtx{B}}$).

The bounds for the minimum eigenvalue follow from the same argument.  In this case,
we must consider negative values of the parameter $\theta$,
but we can transfer the sign to the matrix $\mtx{B}$ and proceed as before.
Together, the bounds on the maximum and minimum eigenvalue lead to
estimates for the spectral norm.

\section{Random Matrices, Exchangeable Pairs, and Kernels} \label{sec:exchange}

Now, we embark on our quest to prove the matrix Efron--Stein inequalities of \secref{main}.
This section outlines some basic concepts from
the theory of exchangeable pairs; cf.~\citep{Stein72,Stein86,Cha07:Steins-Method,Cha08:Concentration-Inequalities}.
Afterward, we explain how these ideas lead to concentration inequalities.

\subsection{Exchangeable Pairs}

In our analysis, the primal concept is an exchangeable pair of random variables.  

\begin{defn}[Exchangeable Pair] \label{def:exchange}
Let $Z$ and $Z'$ be random variables taking values in a Polish space $\metricspace$.
We say that $(Z, Z')$ is an \term{exchangeable pair} when it has the
same distribution as the pair $(Z', Z)$.
\end{defn}

\noindent
In particular, $Z$ and $Z'$ have the same distribution,
and $\Expect f(Z, Z') = \Expect f(Z',Z)$
for every integrable function $f$.

\subsection{Kernel Stein Pairs} \label{sec:pairs}

We are interested in a special class of exchangeable pairs of random matrices.
There must be an antisymmetric bivariate kernel that ``reproduces'' the matrices
in the pair.  This approach is motivated by~\cite{Cha07:Steins-Method}.

\begin{defn}[Kernel Stein Pair] \label{def:K-stein-pair}
Let $(Z, Z')$ be an exchangeable pair of random variables taking values in a Polish space $\metricspace$,
and let $\mtx{\Psi} : \metricspace \to \Sym{d}$ be a measurable function. 
Define the random Hermitian matrices
$$
\mtx{X} \defby \mtx{\Psi}(Z)
\quad\text{and}\quad
\mtx{X}' \defby \mtx{\Psi}(Z').
$$
We say that $(\mtx{X}, \mtx{X}')$ is a \term{kernel Stein pair} if there exists a
bivariate function $\mtx{K} : \metricspace^2 \to \Sym{d}$ for which
\begin{equation} \label{eqn:K-antisymmetry}
\mtx{K}(z,z') = -\mtx{K}(z',z) \quad\text{for all $z, z' \in \metricspace$}
\end{equation}
and
\begin{equation} \label{eqn:K-stein-pair}
\Expect[ \mtx{K}(Z,Z') \condl Z ] = \mtx{X}
	\quad\text{almost surely.}
\end{equation}
When discussing a kernel Stein pair $(\mtx{X}, \mtx{X}')$,
we assume that $\Expect \normsq{\mtx{X}} < \infty$.
We sometimes write \term{$\mtx{K}$-Stein pair} to emphasize the specific kernel $\mtx{K}$.
\end{defn}

The kernel is always centered in the sense that
\begin{equation} \label{eqn:kernel-zero-mean}
\Expect[ \mtx{K}(Z, Z') ] = \mtx{0}.
\end{equation}
Indeed, $\Expect[ \mtx{K}(Z, Z') ] = - \Expect[ \mtx{K}(Z', Z) ] = - \Expect[ \mtx{K}(Z, Z') ]$,
where the first identity follows from antisymmetry and the second follows from exchangeability.

\begin{rem}[Matrix Stein Pairs]
The analysis in~\citep{MackeyJoChFaTr12} is based on a subclass of kernel Stein pairs called \term{matrix Stein pairs}.
A matrix Stein pair $(\mtx{X}, \mtx{X}')$ derived from an auxiliary exchangeable pair $(Z,Z')$ satisfies the stronger condition
\begin{equation} \label{eqn:matrix-stein-pair}
\Expect[ \mtx{X} - \mtx{X}' \bcondl Z ] = \alpha \mtx{X}
\quad\text{for some $\alpha > 0$.}
\end{equation}
That is, a matrix Stein pair is a kernel Stein pair with $\mtx{K}(Z,Z') = \alpha^{-1} (\mtx{X} - \mtx{X}' )$.
Although~\cite{MackeyJoChFaTr12} describe several classes of matrix Stein pairs, most exchangeable pairs of random matrices do not satisfy~\eqref{eqn:matrix-stein-pair}.  Kernel Stein pairs are more common, so they are commensurately more useful.
\end{rem}

\subsection{The Method of Exchangeable Pairs} \label{sec:method-exchange}

Kernel Stein pairs are valuable because they offer a powerful tool for
evaluating moments of a random matrix.  We express this claim in
a fundamental technical lemma, 
which generalizes both~\cite[Eqn.~(6)]{Cha07:Steins-Method}
and~\cite[Lem.~2.3]{MackeyJoChFaTr12}.

\begin{lemma}[Method of Exchangeable Pairs] \label{lem:exchange}
Suppose that $(\mtx{X}, \mtx{X}') \in \Sym{d} \times \Sym{d}$ is a $\mtx{K}$-Stein pair constructed from an auxiliary exchangeable pair $(Z,Z') \in \metricspace^2$.
Let $\mtx{F}: \Sym{d} \rightarrow \Sym{d}$ be a measurable function that satisfies the regularity condition
\begin{equation} \label{eqn:regularity-mep}
\Expect \norm{ \mtx{K}(Z,Z') \, \mtx{F}(\mtx{X})  } < \infty.
\end{equation}
Then
\begin{equation}  \label{eqn:exchange}
 \Expect \left[ \mtx{X} \, \mtx{F}(\mtx{X}) \right]
 	= \frac{1}{2} \Expect \left[ \mtx{K}(Z,Z')(\mtx{F}(\mtx{X}) - \mtx{F}(\mtx{X}') ) \right].
\end{equation}
\end{lemma}
\begin{proof}
\defref{K-stein-pair}, of a kernel Stein pair, implies that
$$
\Expect[ \mtx{X} \, \mtx{F}(\mtx{X}) ]
	= \Expect \big[ \Expect[ \mtx{K}(Z,Z') \condl Z ] \, \mtx{F}(\mtx{X}) \big]
	= \Expect[ \mtx{K}(Z,Z') \, \mtx{F}(\mtx{X}) ],
$$
where we justify the pull-through property of conditional expectation
using the regularity condition~\eqref{eqn:regularity-mep}.
The antisymmetry~\eqref{eqn:K-antisymmetry} of the kernel $\mtx{K}$
delivers the relation
$$
\Expect[ \mtx{K}(Z,Z')\, \mtx{F}(\mtx{X}) ]
	= \Expect[\mtx{K}(Z',Z) \, \mtx{F}(\mtx{X}') ]
	= - \Expect[ \mtx{K}(Z,Z') \, \mtx{F}(\mtx{X}') ].
$$
Average the two preceding displays to reach the identity~\eqref{eqn:exchange}.
\end{proof}

Lemma~\ref{lem:exchange} has several immediate consequences for the structure
of a $\mtx{K}$-Stein pair $(\mtx{X}, \mtx{X}')$ constructed from an auxiliary
exchangeable pair $(Z, Z')$.  First, the matrix $\mtx{X}$ must be centered:
\begin{equation} \label{eqn:X-zero-mean}
\Expect \mtx{X}  = \mtx{0}.
\end{equation}
This result follows from the choice $\mtx{F}(\mtx{X}) = \Id$.

Second, we can develop a bound for the variance of the random matrix $\mtx{X}$.
Since $\mtx{X}$ is centered,
\begin{equation*} \label{eqn:X-variance}
\Var[ \mtx{X} ]
	= \Expect \big[ \mtx{X}^2 \big]
	= \frac{1}{2} \Expect\big[ \real\big(\mtx{K}(Z, Z') (\mtx{X} - \mtx{X}')\big) \big].
\end{equation*}
This claim follows when we apply Lemma~\ref{lem:exchange} with $\mtx{F}(\mtx{X}) = \mtx{X}$
and extract the real part~\eqref{eqn:reim} of the result.
Invoke the matrix inequality~\eqref{eqn:matrix-am-gm} to obtain
\begin{equation} \label{eqn:var-kernel-bd}
\Var[ \mtx{X} ] \psdle
	\frac{1}{4} \Expect\big[ \mtx{K}(Z, Z')^2 + (\mtx{X} - \mtx{X}')^2\big]
\end{equation}
In other words, we can obtain bounds for the variance in terms of the
variance of the kernel $\mtx{K}$ and the variance of $\mtx{X} - \mtx{X}'$.

\subsection{Conditional Variances}

To each kernel Stein pair $(\mtx{X}, \mtx{X}')$, we may associate two random matrices called the \term{conditional variance}  and \term{kernel conditional variance} of $\mtx{X}$.  We will see that $\mtx{X}$ is concentrated around the zero matrix whenever the conditional variance and the kernel conditional variance are both small.

\begin{defn}[Conditional Variances] \label{def:conditional-variance}
Suppose that $(\mtx{X}, \mtx{X}')$ is a $\mtx{K}$-Stein pair, 
constructed from an auxiliary exchangeable pair $(Z, Z')$.
The \term{conditional variance} is the random matrix
\begin{align} \label{eqn:conditional-variance}
\mtx{V}_{\mtx{X}}
	\defby  \frac{1}{2} \Expect \big[(\mtx{X} - \mtx{X}')^2\condl Z \big],
\end{align}
and the \term{kernel conditional variance} is the random matrix
\begin{align} \label{eqn:K-conditional-variance}
\mtx{V}^{\mtx{K}}
	\defby  \frac{1}{2} \Expect \big[\mtx{K}(Z,Z')^2\condl Z \big].
\end{align}
\end{defn}

Because of the bound~\eqref{eqn:var-kernel-bd}, the conditional variances satisfy
$$
\Var[ \mtx{X} ] \psdle \frac{1}{2} \Expect\big[ \mtx{V}_{\mtx{X}} + \mtx{V}^{\mtx{K}} \big],
$$
so it is natural to seek concentration results stated in terms of these quantities.

\section{Polynomial Moments of a Random Matrix}
\label{sec:poly-moment}

We begin by developing a polynomial moment bound for a kernel Stein pair.
This result shows that we can control the expectation of the Schatten $p$-norm
in terms of the conditional variance and the kernel conditional variance.

\begin{thm}[Polynomial Moments for a Kernel Stein Pair] 
\label{thm:bdg-inequality}
Let $(\mtx{X}, \mtx{X}')$ be a $\mtx{K}$-Stein pair
based on an auxiliary exchangeable pair $(Z,Z')$.
For a natural number $p \geq 1$,
assume the regularity conditions
$$
\Expect \pnorm{2p}{\mtx{X}}^{2p} < \infty
\quad\text{and}\quad
\Expect \norm{\mtx{K}(Z,Z')}^{2p} < \infty.
$$
Then, for each $s > 0$,
$$
\big( \Expect \pnorm{2p}{\mtx{X}}^{2p} \big)^{1/(2p)} 
	\leq \sqrt{2p-1} \left( \Expect \pnorm{p}{\frac{1}{2} \big(s \, \mtx{V}_{\mtx{X}} + s^{-1} \, \mtx{V}^{\mtx{K}} \big) }^p \right)^{1/(2p)}.
$$
The symbol $\pnorm{p}{\cdot}$ refers to the Schatten $p$-norm \eqref{eqn:schatten-norm},
and the conditional variances $\mtx{V}_{\mtx{X}}$ and $\mtx{V}^{\mtx{K}}$
are defined in~\eqref{eqn:conditional-variance} and~\eqref{eqn:K-conditional-variance}.
\end{thm}

\noindent
We establish this result, which holds equally for infinite dimensional operators $\mtx{X}$, in the remainder of this section.
The pattern of argument is similar to the proofs
of~\cite[Thm.~3.14]{Cha08:Concentration-Inequalities}
and~\cite[Thm.~7.1]{MackeyJoChFaTr12}, but we require
a nontrivial new matrix inequality.

\subsection{The Polynomial Mean Value Trace Inequality}

The main new ingredient in the proof of Theorem~\ref{thm:bdg-inequality}
is the following matrix trace inequality.

\begin{lemma}[Polynomial Mean Value Trace Inequality] \label{lem:pmvti} 
For all matrices $\mtx{A}$, $\mtx{B}$, $\mtx{C} \in\Sym{d}$, all integers $q \geq 1$, and all $s >0$, it holds that
\begin{align*}
\abs{\trace \left[\mtx{C} (\mtx{A}^{q} - \mtx{B}^{q} )\right]} \leq  
	\frac{q}{4} \trace \big[(s \, (\mtx{A}-\mtx{B})^2+s^{-1} \, \mtx{C}^2)(\abs{\mtx{A}}^{q-1} + \abs{\mtx{B}}^{q-1}) \big].
\end{align*}
\end{lemma}

\noindent
Lemma~\ref{lem:pmvti} improves on the estimate~\cite[Lem.~3.4]{MackeyJoChFaTr12},
which drives concentration inequalities for matrix Stein pairs.
Since the result does not have any probabilistic content,
we defer the proof until Appendix~\ref{sec:proof-pmvti}.

\subsection{Proof of Theorem~\ref{thm:bdg-inequality}}
\label{sec:bdg-proof}

The argument follows the same lines as the proof of~\cite[Thm.~7.1]{MackeyJoChFaTr12},
so we pass lightly over certain details.
Let us examine the quantity of interest:
\begin{align*}
E \defby \Expect \pnorm{2p}{\mtx{X}}^{2p}
	= \Expect \trace \abs{\mtx{X}}^{2p}
	=  \Expect \trace \big[ \mtx{X}\cdot{\mtx{X}}^{2p-1}\big]
\end{align*}
where $\cdot$ denotes the usual matrix product.
To apply the method of exchangeable pairs, Lemma~\ref{lem:exchange},
we first check the regularity condition~\eqref{eqn:regularity-mep}:
\begin{align*} 
\Expect \smnorm{}{ \mtx{K}(Z,Z') \cdot {\mtx{X}}^{2p-1} }
	&\leq \Expect \big( \norm{ \mtx{K}(Z,Z')}  \norm{ \mtx{X} }^{2p-1} \big) \\
	&\leq \big(\Expect \norm{ \mtx{K}(Z,Z')}^{2p} \big)^{1/(2p)} \big(\Expect \norm{ \mtx{X} }^{2p} \big)^{(2p-1)/(2p)} 
	< \infty,
\end{align*}
where we have applied \Holder's inequality for expectation
and the fact that the spectral norm is dominated by the Schatten $2p$-norm.
Thus, we may invoke Lemma~\ref{lem:exchange} with $\mtx{F}(\mtx{X}) = {\mtx{X}}^{2p - 1}$ to reach
\begin{align*}
E = \frac{1}{2} \Expect \trace\big[ \mtx{K}(Z,Z') \cdot
	\big( {\mtx{X}}^{2p-1} - ({\mtx{X}'})^{2p-1} \big) \big].
\end{align*}
Next, fix a parameter $s > 0$.  Apply the polynomial mean value trace inequality, Lemma~\ref{lem:pmvti}, with $q=2p-1$ to obtain the estimate
\begin{align*}
E 
	&\leq \frac{2p-1}{8} \Expect \trace\big[\big(s \, (\mtx{X}-\mtx{X}')^2+s^{-1} \mtx{K}(Z,Z')^2 \big)
		\cdot \big(\mtx{X}^{2p-2} + (\mtx{X}')^{2p-2} \big) \big] \\
	&= \frac{2p-1}{4} \Expect \trace\big[ \big(s \, (\mtx{X}-\mtx{X}')^2+ s^{-1} \mtx{K}(Z,Z')^2 \big)
		\cdot\mtx{X}^{2p-2} \big] \\
	&= (2p-1) \Expect \trace\left[\frac{1}{2}\big(s \, \mtx{V}_{\mtx{X}}+s^{-1} \mtx{V}^{\mtx{K}}\big)\cdot \mtx{X}^{2p-2}\right].
\end{align*}
The second line follows from the fact that $(\mtx{X}, \mtx{X}')$ is an exchangeable pair,
and the third line depends on the definitions~\eqref{eqn:conditional-variance}
and \eqref{eqn:K-conditional-variance} of the conditional variances.
We have used the regularity condition $\Expect \pnorm{2p}{\mtx{X}}^{2p} < \infty$
to justify the pull-through property of conditional expectation.

Now, we apply H{\"o}lder's inequality for the trace followed by H{\"o}lder's inequality
for the expectation.  These steps yield
\begin{align*}
E &\leq (2p - 1) \left( \Expect \pnorm{p}{ \frac{1}{2} \big(s\, \mtx{V}_{\mtx{X}} + s^{-1}\, \mtx{V}^{\mtx{K}} \big) }^p
	\right)^{1/p} \left( \Expect \pnorm{2p}{ \mtx{X} }^{2p} \right)^{(p-1)/p} \\
	&= (2p - 1) \left( \Expect \pnorm{p}{ \frac{1}{2} \big(s\, \mtx{V}_{\mtx{X}} + s^{-1}\, \mtx{V}^{\mtx{K}} \big) }^p
	\right)^{1/p} E^{(p-1)/p}.
\end{align*}
Solve this algebraic identity for $E$ to determine that
$$
E^{1/(2p)} \leq \sqrt{2p - 1} \left( \Expect \pnorm{p}{ \frac{1}{2} \big(s \,\mtx{V}_{\mtx{X}}
	+ s^{-1}\, \mtx{V}^{\mtx{K}} \big) } \right)^{1/(2p)}.
$$
This completes the proof of Theorem~\ref{thm:bdg-inequality}.

\section{Constructing a Kernel via Markov Chain Coupling}
\label{sec:kernels}

Theorem~\ref{thm:bdg-inequality} is one of the main steps toward
the polynomial Efron--Stein inequality for random matrices.
To reach the latter result, we need to develop an explicit construction
for the kernel Stein pair along with concrete bounds for
the conditional variance.  We present this material in the current section, 
and we establish the Efron--Stein bound in Section~\ref{sec:poly-efronstein}.
The analysis leading to exponential concentration inequalities is somewhat more involved.
We postpone these results until Section~\ref{sec:concentration-bdd}.

\subsection{Overview}

For a random matrix $\mtx{X}$ that is presented as part of a kernel
Stein pair, Theorem~\ref{thm:bdg-inequality} provides strong bounds
on the polynomial moments in terms of the conditional variances.
To make this result effective, we need to address several more
questions.

First, given an exchangeable pair of random matrices, we can ask
whether it is possible to equip the pair with a kernel
that satisfies~\eqref{eqn:K-stein-pair}.  In fact, there is a
general construction that works whenever the exchangeable
pair is suitably ergodic.  This method depends on 
an idea~\cite[Sec.~4.1]{Cha08:Concentration-Inequalities}
that ultimately relies on an observation of Stein, cf.~\citep{Stein86}.
We describe this approach in Sections~\ref{sec:couplings} and \ref{sec:poisson}.

Second, we can ask whether there is a mechanism for bounding the
conditional variances in terms of simpler quantities.
We have developed some new tools for performing these estimates.
These methods appear in Sections~\ref{sec:cond-var-bd-1} and~\ref{sec:cond-var-bd-2}.

\subsection{Kernel Couplings} \label{sec:couplings}

Stein noticed that each exchangeable pair $(Z,Z')$
of $\metricspace$-valued random variables yields a reversible Markov chain
with a symmetric transition kernel $P$ given by
$$
Pf(z) \defby \Expect[ f(Z') \condl Z = z]
$$
for each function $f : \metricspace \to \R$ that satisfies
$\Expect \abs{f(Z)} < \infty$.
In other words, for any initial value $Z_{(0)} \in \metricspace$,
we can construct a Markov chain
$$
Z_{(0)} \to Z_{(1)} \to Z_{(2)} \to Z_{(3)} \to \cdots
$$
where $\Expect[ f(Z_{(i+1)}) \condl Z_{(i)} ] = Pf(Z_{(i)})$
for each integrable function $f$.  This requirement suffices
to determine the distribution of each $Z_{(i+1)}$.

When the chain $(Z_{(i)})_{i \geq 0}$ is ergodic enough,
we can explicitly construct a kernel that
satisfies~\eqref{eqn:K-stein-pair} for any exchangeable
pair of random matrices constructed from the auxiliary
exchangeable pair $(Z,Z')$.  To explain this idea, we
adapt a definition
from~\cite[Sec.~4.1]{Cha08:Concentration-Inequalities}.

\begin{defn}[Kernel Coupling] \label{def:kernel-coupling}
Let $(Z,Z') \in \metricspace^2$ be an exchangeable pair. 
Let $(Z_{(i)})_{i\geq 0}$ and $(Z'_{(i)})_{i\geq 0}$
be two Markov chains with arbitrary initial values, each
evolving according to the transition kernel $P$ induced
by $(Z,Z')$.
We call $(Z_{(i)},Z'_{(i)})_{i\geq 0}$ a \emph{kernel coupling} for $(Z,Z')$ if
\begin{align} \label{eqn:kernel-coupling}
Z_{(i)} \indep Z_{(0)}' \bcondl Z_{(0)}\quad\text{and}\quad Z'_{(i)} \indep Z_{(0)} \bcondl Z'_{(0)} \quad \text{for all $i$.}
\end{align}
The notation $U \indep V \condl W$ means $U$ and $V$ are independent conditional on $W$.
\end{defn}

For an example of kernel coupling, consider the simple random walk on the hypercube
$\{\pm 1\}^n$ where two vertices are neighbors when they differ in exactly one coordinate.
We can start two random walks at two different locations on the
cube.  At each step, we select a uniformly random coordinate from $\{1, \dots, n\}$
and a uniformly random value from $\{\pm 1\}$.  We update \emph{both} of the walks
by replacing the \emph{same} chosen coordinate with the \emph{same}
chosen value.  The two walks arrive at the same vertex (i.e., they \term{couple})
as soon as we have updated each coordinate at least once.

\subsection{Kernel Stein Pairs from the Poisson Equation}
\label{sec:poisson}

Chatterjee~\cite[Sec.4.1]{Cha08:Concentration-Inequalities} observed
that it is often possible to construct a kernel coupling
by solving the Poisson equation for the Markov chain
with transition kernel $P$.

\begin{prop} \label{lem:kernel-coupling}
	Let $(Z_{(i)},Z'_{(i)})_{i\geq 0}$ be a kernel coupling for an exchangeable pair $(Z,Z') \in \metricspace^2$.
	Let $\mtx{\Psi} : \metricspace \to \Sym{d}$ be a bounded, measurable function
	with $\Expect\mtx{\Psi}(Z) = \zeromtx.$
	Suppose there is a positive constant $L$ for which
	\begin{equation}\label{eqn:coupling-premise}
	\sum\nolimits_{i=0}^\infty \norm{ \Expect\big[\mtx{\Psi}(Z_{(i)}) - \mtx{\Psi}(Z'_{(i)})
		\bcondl Z_{(0)}=z,Z'_{(0)}=z'\big]} \leq L \quad\text{for all $z, z' \in \metricspace$}.
	\end{equation}
	Then $(\mtx{\Psi}(Z), \mtx{\Psi}(Z'))$ is a $\mtx{K}$-Stein pair
	with kernel 
	\begin{align} \label{eqn:K-construct}
	\mtx{K}(z,z') \defby \sum\nolimits_{i=0}^\infty \Expect\big[\mtx{\Psi}(Z_{(i)})
		- \mtx{\Psi}(Z'_{(i)})\bcondl Z_{(0)}=z,Z'_{(0)}=z'\big].
	\end{align}
\end{prop}

\noindent
The proof of this result is identical with that of \cite[Lem.~4.1 and 4.2]{Cha08:Concentration-Inequalities}, which establishes Proposition~\ref{lem:kernel-coupling} in the scalar setting.  We omit the details.

\begin{rem}[Regularity]
Proposition~\ref{lem:kernel-coupling} holds for
functions $\mtx{\Psi}$ that satisfy conditions weaker
than boundedness.  We focus on the simplest case to
reduce the technical burden.
\end{rem}

\subsection{Bounding the Conditional Variances I}
\label{sec:cond-var-bd-1}

The construction described in Proposition~\ref{lem:kernel-coupling}
is valuable because it leads to an explicit description of the kernel.
In many examples, this formula allows us to develop a succinct bound on the conditional variances.
We encapsulate the required calculations in a technical lemma.

\begin{lemma} \label{lem:conditional-variance-bound}
Instate the notation and hypotheses of Proposition~\ref{lem:kernel-coupling},
and define the kernel Stein pair $(\mtx{X}, \mtx{X}') = (\mtx{\Psi}(Z), \mtx{\Psi}(Z'))$.
For each $i = 0, 1, 2, \dots$, assume that
\begin{align} \label{eqn:K-term-bound}
\Expect\Big[ \Big( \Expect\big[\mtx{\Psi}(Z_{(i)}) - \mtx{\Psi}(Z'_{(i)})\bcondl Z_{(0)} = Z, Z'_{(0)} = Z' \big] \Big)^2 \bcondl Z \Big] 
	\psdle \beta_i^2 \, \mtx{\Gamma}_i
\end{align}
where $\beta_i$ is a nonnegative number and $\mtx{\Gamma}_i \in \Sym{d}$ is a random matrix.
Then the conditional variance~\eqref{eqn:conditional-variance} and kernel
conditional variance~\eqref{eqn:K-conditional-variance} satisfy
\begin{align*} 
\mtx{V}_{\mtx{X}} 
	\psdle \frac{1}{2} \beta_0^2 \, \mtx{\Gamma}_0
	\quad\text{and}\quad
\mtx{V}^{\mtx{K}} 
	\psdle \frac{1}{2} \left(\sum\nolimits_{j=0}^\infty \beta_j \right) \sum\nolimits_{i=0}^\infty \beta_i\, \mtx{\Gamma}_i.
\end{align*}
\end{lemma}

\begin{proof}
By a continuity argument, we may assume that $\beta_i > 0$ for each index $i$.
Write
$$
\mtx{Y}_i \defby \Expect\big[\mtx{\Psi}(Z_{(i)}) - \mtx{\Psi}(Z'_{(i)})\bcondl Z_{(0)} = Z,Z'_{(0)} = Z' \big].
$$
The definition~\eqref{eqn:conditional-variance} of the
conditional variance $\mtx{V}_{\mtx{X}}$ immediately implies
\begin{align*}
\mtx{V}_{\mtx{X}}
	= \frac{1}{2} \Expect \big[ (\mtx{X} - \mtx{X}')^2 \bcondl Z \big]
	= \frac{1}{2}
	\Expect\big[\mtx{Y}_0^2 \bcondl Z \big] 
	\psdle\frac{1}{2} \beta_0^2 \, \mtx{\Gamma}_0.
\end{align*}
The semidefinite relation follows from the hypothesis~\eqref{eqn:K-term-bound}.

According to the definition~\eqref{eqn:K-conditional-variance}
of the kernel conditional variance $\mtx{V}^{\mtx{K}}$ and
the kernel construction~\eqref{eqn:K-construct}, we have
$$
\begin{aligned}
\mtx{V}^{\mtx{K}} = \frac{1}{2} \Expect \big[ \mtx{K}(Z,Z') \condl Z \big]
	&= \frac{1}{2} \Expect\left[ \left(\sum\nolimits_{i=0}^\infty \mtx{Y}_i \right)^2 \bcondl Z \right] \\
	&= \frac{1}{2}\sum\nolimits_{i=0}^\infty \sum\nolimits_{j=0}^\infty 
	\Expect\big[\real(\mtx{Y}_i\mtx{Y}_j) \bcondl Z\big].
\end{aligned}
$$
The semidefinite bound~\eqref{eqn:matrix-am-gm} for the real part of a product implies that
\begin{align*}
\mtx{V}^{\mtx{K}} 
	&\psdle \frac{1}{2}\sum\nolimits_{i=0}^\infty \sum\nolimits_{j=0}^\infty 
	\frac{1}{2}\left(\frac{\beta_j}{\beta_i}\Expect\big[\mtx{Y}_i^2\bcondl Z\big]
		+ \frac{\beta_i}{\beta_j}\Expect\big[\mtx{Y}_j^2\bcondl Z\big]\right) \\
	&\psdle \frac{1}{2}\sum\nolimits_{i=0}^\infty \sum\nolimits_{j=0}^\infty 
	\frac{1}{2}\left(\frac{\beta_j}{\beta_i}\beta_i^2 \, \mtx{\Gamma}_i + \frac{\beta_i}{\beta_j}\beta_j^2 \,\mtx{\Gamma}_j \right) \\
	&=\frac{1}{2} \left(\sum\nolimits_{j=0}^\infty \beta_j \right) \sum\nolimits_{i=0}^\infty \beta_i \, \mtx{\Gamma}_i.
\end{align*}
The second relation depends on the hypothesis \eqref{eqn:K-term-bound}.
\end{proof}

\subsection{Bounding the Conditional Variances II}
\label{sec:cond-var-bd-2}

The random matrices $\mtx{\Gamma}_i$ that arise in Lemma~\ref{lem:conditional-variance-bound}
often share a common form.  We can use this property to obtain a succinct bound for the conditional
variance expression that appears in Theorem~\ref{thm:bdg-inequality}.
This reduction allows us to establish Efron--Stein inequalities.

\begin{lemma} \label{lem:convex-conditional-variance-bound}
Instate the notation and hypotheses of Lemma~\ref{lem:conditional-variance-bound}.
Suppose
\begin{equation} \label{eqn:cond-var-hyp-ii}
\mtx{\Gamma}_i = \Expect\big[ \mtx{W}_{(i)} \condl Z \big]
\quad\text{where}\quad
\mtx{W}_{(i)} \sim \mtx{\Gamma}_0
\quad\text{for each $i \geq 1$.}
\end{equation}
Then, for each weakly increasing and convex function $f : \R_+ \to \R$,
$$
\Expect \trace f\big(\beta_0^{-2} \, \mtx{V}_{\mtx{X}}
	+ (\textstyle\sum\nolimits_{i=0}^\infty \beta_i)^{-2} \, \mtx{V}^{\mtx{K}} \big)
	\leq \Expect \trace f\left(\mtx{\Gamma}_0\right).
$$
\end{lemma}

\begin{proof}
Abbreviate $B = \sum_{i=0}^\infty \beta_i$.  Lemma~\ref{lem:conditional-variance-bound} provides that
\begin{align*} 
\mtx{V}_{\mtx{X}} 
	\psdle \frac{1}{2} \beta_0^2 \, \mtx{\Gamma}_0
\qtext{and}	
\mtx{V}^{\mtx{K}} 
	\psdle \frac{1}{2} B \sum\nolimits_{i=0}^\infty \beta_i \, \mtx{\Gamma}_i.
\end{align*}
Since $f$ is weakly increasing and convex on $\R_+$, the function $\trace f : \Sym{d}_+ \to \R$
is weakly increasing~\eqref{eqn:trace-monotone} and convex~\eqref{eqn:trace-convex}.  Therefore,
\begin{align*}
\Expect \trace f\big(\beta_0^{-2} \, \mtx{V}_{\mtx{X}} + B^{-2} \, \mtx{V}^{\mtx{K}} \big)
	&\leq \Expect \trace f\left( \frac{1}{2} \mtx{\Gamma}_0 + \frac{1}{2B} \sum\nolimits_{i=0}^\infty \beta_i \,\mtx{\Gamma}_i \right) \\
	&\leq \frac{1}{2} \Expect \trace f( \mtx{\Gamma}_0 )
	+ \frac{1}{2B} \sum\nolimits_{i=0}^\infty \beta_i \Expect \trace f(\mtx{\Gamma}_i).
\end{align*}
In view of~\eqref{eqn:cond-var-hyp-ii}, Jensen's inequality and the tower property together yield
\begin{align*}
\Expect \trace f( \mtx{\Gamma}_i )
	= \Expect \trace f\big( \Expect\big[ \mtx{W}_{(i)} \condl Z \big] \big)
	\leq \Expect \trace f( \mtx{W}_{(i)} )
	= \Expect \trace f( \mtx{\Gamma}_0) .
\end{align*}
Combine the latter two displays to complete the argument.
\end{proof}

\section{The Polynomial Efron--Stein Inequality for a Random Matrix}
\label{sec:poly-efronstein}

We are now prepared to establish the polynomial Efron--Stein
inequality, Theorem~\ref{thm:mxmoment}.
We retain the notation and hypotheses from Section~\ref{sec:setup},
and we encourage the reader to review this material before continuing.
The proof is divided into two parts.  First, we assume that
the random matrix is bounded so that the kernel coupling tools
apply.  Then, we use a truncation argument to remove the boundedness
assumption.

\subsection{A Kernel Coupling for a Vector of Independent Variables}
\label{sec:my-kernel}

We begin with the construction of an exchangeable pair.
Recall that $Z := (Z_1, \dots, Z_n) \in \metricspace$ is a vector of
mutually independent random variables.  For each coordinate $j$,
$$
Z^{(j)} := (Z_1, \dots, \widetilde{Z}_j, \dots, Z_n) \in \metricspace
$$
where $\widetilde{Z}_j$ is an independent copy of $Z_j$.
Form the random vector
\begin{equation} \label{eqn:Z'}
Z' := Z^{(J)}
\quad\text{where}\quad
J \sim \textsc{uniform}\{1, \dots, n\}.
\end{equation}
We may assume that $J$ is drawn independently from $Z$.
It follows that $(Z, Z')$ is exchangeable.

Next, we build an explicit kernel coupling $(Z_{(i)}, Z'_{(i)})_{i \geq 0}$
for the exchangeable pair $(Z,Z')$.
The Markov chains may take arbitrary initial values $Z_{(0)}$ and $Z'_{(0)}$.
For each time $i \geq 1$, we let both chains evolve via the same random choice:

\begin{enumerate}
\item	Independent of prior choices, draw a coordinate $J_i \sim \textsc{uniform}\{1, \dots, n\}$.

\item	Draw an independent copy $\widetilde{Z}_{(i)}$ of $Z$.

\item	Form $Z_{(i)}$ by replicating $Z_{(i-1)}$ and then replacing the $J_{i}$-th coordinate
with the $J_{i}$-th coordinate of $\widetilde{Z}_{(i)}$.

\item	Form $Z'_{(i)}$ by replicating $Z'_{(i-1)}$ and then replacing the $J_{i}$-th coordinate
with the $J_{i}$-th coordinate of $\widetilde{Z}_{(i)}$.
\end{enumerate}

\noindent
By construction, $(Z_{(i)}, Z'_{(i)})_{i \geq 0}$ satisfies the kernel coupling
property~\eqref{eqn:kernel-coupling}.
This coupling is drawn from~\cite[Sec.~4.1]{Cha08:Concentration-Inequalities}.
Note that this is just a glorification of the hypercube example in Section~\ref{sec:couplings}!

\subsection{A Kernel Stein Pair}

Let $\mtx{H} : \metricspace \to \Sym{d}$ be a bounded, measurable function.
Construct the random matrices
\begin{equation} \label{eqn:my-stein}
\mtx{X} := \mtx{H}(Z) - \Expect \mtx{H}(Z)
\quad\text{and}\quad
\mtx{X}' := \mtx{H}(Z') - \Expect \mtx{H}(Z).
\end{equation}
To verify that $(\mtx{X}, \mtx{X}')$ is a kernel Stein pair,
we use Lemma~\ref{lem:kernel-coupling} to construct a kernel.
For all $z, z' \in \metricspace$,
\begin{equation} \label{eqn:es-kernel}
\mtx{K}(z, z') := \sum_{i=0}^\infty
	\Expect \big[ \mtx{H}(Z_{(i)}) - \mtx{H}(Z'_{(i)}) \bcondl Z_{(0)} = z, Z'_{(0)} = z' \big].
\end{equation}
To verify the regularity condition for the lemma, notice that the two chains couple
as soon as we have refreshed all $n$ coordinates.  According to the analysis of
the coupon collector problem~\cite[Sec.~2.2]{LPW10:Markov-Chains}, the expected
coupling time is bounded by $n(1 + \log n)$.  Since $\norm{\mtx{H}(Z)}$ is bounded,
the hypothesis~\eqref{eqn:coupling-premise} is in force.

\subsection{The Evolution of the Kernel Coupling}
\label{eqn:my-evol}

Draw a realization $(Z, Z')$ of the exchangeable pair,
and write $J$ for the coordinate where $Z$ and $Z'$ differ.
Let $(Z_{(i)}, Z'_{(i)})_{i \geq 0}$ be the kernel coupling
described in the last section, starting at $Z_{(0)} = Z$
and $Z'_{(0)} = Z'$.  Therefore, the initial value of the
kernel coupling is a pair of vectors that differ in precisely one coordinate.
Because of the coupling construction,
$$
\mtx{H}(Z_{(i)}) - \mtx{H}(Z'_{(i)})=(\mtx{H}(Z_{(i)}) - \mtx{H}(Z'_{(i)}))\cdot \II[J \notin \{ J_1, \dots, J_i \}] .
$$
The operator Schwarz inequality~\cite[Eqn.~(3.19)]{Bha07:Positive-Definite} implies that
\begin{align} \label{eqn:es-bd-1}
\Big( \Expect \big[ & \mtx{H}(Z_{(i)}) - \mtx{H}(Z'_{(i)}) \bcondl Z, Z' \big] \Big)^2 \notag \\
	&= \Big( \Expect \big[
		(\mtx{H}(Z_{(i)}) - \mtx{H}(Z'_{(i)})) \cdot \II[J \notin \{ J_1, \dots, J_i \}  ] \bcondl Z, Z'\big] \Big)^2 \notag \\
	&\psdle \Expect \big[
		(\mtx{H}(Z_{(i)}) - \mtx{H}(Z'_{(i)}))^2\bcondl Z, Z'\big] \cdot \Expect \big[
		\II[J \notin \{ J_1, \dots, J_i \}]\bcondl Z, Z'\big] \notag \\
	&= (1 - 1/n)^{i} \cdot \Expect \Big[
		\big( \mtx{H}(Z_{(i)}) - \mtx{H}(Z'_{(i)}) \big)^2 \bcondl Z, Z'\Big].
\end{align}
Take the conditional expectation with respect to $Z$, and invoke the tower property to reach
\begin{align} \label{eqn:es-bd-2}
\Expect \Big[ \Big( \Expect \big[ & \mtx{H}(Z_{(i)}) - \mtx{H}(Z'_{(i)}) \bcondl Z, Z' \big] \Big)^2 \bcondl Z \Big]
	\notag \\
	& \psdle (1 - 1/n)^{i} \cdot \Expect \Big[
		\big( \mtx{H}(Z_{(i)}) - \mtx{H}(Z'_{(i)}) \big)^2 \bcondl Z\Big].
\end{align}

\subsection{Conditional Variance Bounds}

To obtain a bound for the expression~\eqref{eqn:es-bd-2} that satisfies the prerequisites of \lemref{convex-conditional-variance-bound},
we will replace $Z'_{(i)}$ with a variable $Z^*_{(i)}$ that satisfies
$$
(Z_{(i)},Z^*_{(i)}) \sim (Z,Z')
\qtext{and}
Z^*_{(i)} \indep Z \mid Z_{(i)}.
$$
For $i\ge 0$, define $Z^*_{(i)}$ as being equal to $Z_{(i)}$ everywhere except in coordinate $J$, where it equals $Z'_{J}$. 
Since $(J, Z'_{J})\indep Z\mid Z_{(i)}$, we have our desired conditional independence.
Moreover, this definition ensures that $Z^*_{(i)}=Z'_{(i)}$ whenever $J\notin \{J_1,\ldots,J_i\}$. 
Therefore,
$$
\Expect \Big[\big( \mtx{H}(Z_{(i)}) - \mtx{H}(Z'_{(i)}) \big)^2 \bcondl Z\Big]
\psdle 
\Expect \Big[\big( \mtx{H}(Z_{(i)}) - \mtx{H}(Z^*_{(i)}) \big)^2 \bcondl Z\Big].
$$
Consequently, the hypothesis~\eqref{eqn:K-term-bound} of Lemma~\ref{lem:conditional-variance-bound} is valid with
\begin{equation} \label{eqn:my-gamma-def}
\mtx{\Gamma}_i := \Expect \Big[
	\big( \mtx{H}(Z_{(i)}) - \mtx{H}(Z^*_{(i)}) \big)^2 \bcondl Z \Big]
\end{equation}
and $\beta_i := (1 - 1/n)^{i/2}$.

Now, let us have a closer look at the form of $\mtx{\Gamma}_i$.
The tower property and conditional independence of $(Z^*_{(i)}, Z)$ imply that 
\begin{align*}
\mtx{\Gamma}_i
	&= \Expect \Big[ \Expect \Big[ \big( \mtx{H}(Z_{(i)}) - \mtx{H}(Z^*_{(i)}) \big)^2
	\bcondl Z_{(i)}, Z \Big] \bcondl Z \Big]\\
	&=\Expect \Big[ \Expect \Big[ \big( \mtx{H}(Z_{(i)}) - \mtx{H}(Z^*_{(i)}) \big)^2
	\bcondl Z_{(i)}\Big] \bcondl Z \Big].
\end{align*}
Since
\begin{equation} \label{eqn:my-gamma0}
\mtx{\Gamma}_0 = \Expect \Big[ \big(\mtx{H}(Z) - \mtx{H}(Z') \big)^2 \bcondl Z \Big],
\end{equation}
we can express the latter observation as
$$
\mtx{\Gamma}_i = \Expect \big[ \mtx{W}_{(i)} \bcondl Z \big]
\quad\text{where}\quad
\mtx{W}_{(i)} \sim \mtx{\Gamma}_0
$$
by setting
$$
\mtx{W}_{(i)}:=\Expect \Big[ \big( \mtx{H}(Z_{(i)}) - \mtx{H}(Z^*_{(i)}) \big)^2\bcondl Z_{(i)}\Big].
$$
This is the second hypothesis required by Lemma~\ref{lem:convex-conditional-variance-bound}.

\subsection{The Polynomial Efron--Stein Inequality: Bounded Case}
\label{sec:poly-es-proof}

We are prepared to prove
the polynomial Efron--Stein inequality, Theorem~\ref{thm:mxmoment},
for a \emph{bounded} random matrix $\mtx{X}$ of the form~\eqref{eqn:my-stein}.

Let $p$ be a natural number.  Since $(\mtx{X}, \mtx{X}')$ is
a kernel Stein pair, Theorem~\ref{thm:bdg-inequality}
provides that for any $s>0$,
\begin{equation} \label{eqn:poly-es-pf-1}
\left( \Expect \pnorm{2p}{ \mtx{X} }^{2p} \right)^{1/(2p)}
	\leq \sqrt{2p - 1}
	\left( \Expect \pnorm{p}{ \frac{1}{2} \big(s\, \mtx{V}_{\mtx{X}}
	+ s^{-1} \,\mtx{V}^{\mtx{K}} \big) }^p \right)^{1/(2p)}.
\end{equation}
The regularity condition holds because both
the random matrix $\mtx{X}$ and the kernel $\mtx{K}$ are bounded.

Rewrite the Schatten $p$-norm in terms of the trace:
\begin{equation} \label{eqn:poly-es-pf-2}
\Expect \pnorm{p}{ \frac{1}{2} \big(s \, \mtx{V}_{\mtx{X}} + s^{-1}\, \mtx{V}^{\mtx{K}} \big) }^p
	= \Expect \trace \left[ \frac{s}{2} \big( \mtx{V}_{\mtx{X}} + s^{-2} \, \mtx{V}^{\mtx{K}} \big) \right]^p.
\end{equation}
This expression has the form required by Lemma~\ref{lem:convex-conditional-variance-bound}.
Indeed, the function $t \mapsto (st/2)^p$ is weakly increasing and convex on $\R_+$.
Furthermore, we may choose $\beta_0 = 1$ and
\begin{equation*}\label{eq:sdef}
s:=\sum_{i=0}^\infty \beta_i =\left(1-\left(1-\frac{1}{n}\right)^{-1/2}\right)^{-1}< 2n.
\end{equation*}
Lemma~\ref{lem:convex-conditional-variance-bound} now delivers the bound
\begin{equation} \label{eqn:poly-es-pf-3}
\Expect \trace \left[ \frac{s}{2} \big( \mtx{V}_{\mtx{X}} + s^{-2} \, \mtx{V}^{\mtx{K}} \big) \right]^p
	\leq \Expect \trace \left[ \frac{1}{2}s\, \mtx{\Gamma}_0 \right]^{p}
	\le \Expect \pnorm{p}{ 2\cdot \frac{n}{2}\, \mtx{\Gamma}_0 }^p.
\end{equation}

Next, we observe that the random matrix $\half n\, \mtx{\Gamma}_0$ coincides
with the variance proxy $\mtx{V}$ defined in~\eqref{eq:Vdef}.  Indeed,
\begin{align} \label{eq:my-Vdef}
\frac{1}{2} n\, \mtx{\Gamma}_0
	&= \frac{1}{2} n\,\Expect \left[ \big( \mtx{H}(Z) - \mtx{H}(Z') \big)^2 \bcondl Z \right] \notag \\
	&= \frac{1}{2} \sum_{j=1}^n \Expect \left[ \big( \mtx{H}(Z) - \mtx{H}(Z^{(j)}) \big)^2 \bcondl Z \right] \notag \\
	&= \frac{1}{2} \sum_{j=1}^n \Expect \left[ \big( \mtx{X} - \mtx{X}^{(j)} \big)^2 \bcondl Z \right]
	= \mtx{V}.
\end{align}
The first identity is~\eqref{eqn:my-gamma0}.  The second follows from
the definition~\eqref{eqn:Z'} of $Z'$.  The last line harks back
to the definition~\eqref{eqn:Xj} of $\mtx{X}^{(j)}$ and the
variance proxy~\eqref{eq:Vdef}.

Sequence the displays~\eqref{eqn:poly-es-pf-1},~\eqref{eqn:poly-es-pf-2},~\eqref{eqn:poly-es-pf-3},
and~\eqref{eq:my-Vdef} to reach
\begin{equation} \label{eqn:poly-es-bdd}
\left( \Expect \pnorm{2p}{ \mtx{X} }^{2p} \right)^{1/(2p)}
	\leq \sqrt{2(2p - 1)} \left( \Expect \pnorm{p}{ \mtx{V} }^p \right)^{1/(2p)}
	\quad\text{when $\mtx{X}$ is bounded.}
\end{equation}
This complete the proof of Theorem~\ref{thm:mxmoment}
under the assumption that $\mtx{X}$ is bounded.

\subsection{The Polynomial Efron--Stein Inequality: General Case}

Finally, we establish Theorem~\ref{thm:mxmoment}
by removing the stipulation that $\mtx{X}$ is bounded from~\eqref{eqn:poly-es-bdd}.
We assume $\Expect \pnorm{2p}{\mtx{X}}^{2p} < \infty$ so that
there is something to prove.

Let $R > 0$ be a parameter, and introduce the truncated random matrix
$$
\mtx{X}_R = \mtx{X} \cdot \mathds{1}\{\pnorm{2p}{\mtx{X}} \leq R \}
$$
where $\mathds{1}$ denotes the 0--1 indicator of an event.  Apply~\eqref{eqn:poly-es-bdd}
to the bounded and centered random matrix $\mtx{X}_R - \Expect \mtx{X}_R$ to obtain
\begin{equation} \label{eqn:poly-es-trunc}
\left( \Expect \pnorm{2p}{ \mtx{X}_R- \Expect \mtx{X}_R}^{2p} \right)^{1/(2p)}
	\leq \sqrt{2(2p - 1)} \left( \Expect \pnorm{p}{ \mtx{V}_R }^p \right)^{1/(2p)}
\end{equation}
where
$$
\mtx{V}_R := \frac{1}{2} \sum_{j=1}^n \Expect\big[ \big( \mtx{X}_R - \mtx{X}_R^{(j)} \big)^2 \bcondl Z \big]
\quad\text{and}\quad
\mtx{X}^{(j)}_R := \mtx{X}^{(j)} \mathds{1}\{\pnorm{2p}{\smash{\mtx{X}^{(j)}}} \leq R \}.
$$
To complete the argument, we just need to take the limits as $R \to \infty$.

For the left-hand side of~\eqref{eqn:poly-es-trunc}, first observe that
$\pnorm{2p}{\mtx{X}_R} \uparrow \pnorm{2p}{\mtx{X}}$ everywhere.
Therefore,
\begin{equation} \label{eqn:XR-lim}
\Expect \pnorm{2p}{\mtx{X}_R}^{2p}
	\uparrow \Expect \pnorm{2p}{\mtx{X}}^{2p}
\end{equation}
because of the monotone convergence theorem. Now we are left to show that $\Expect \pnorm{2p}{\mtx{X}_R-\Expect \mtx{X}_R}^{2p} -\Expect \pnorm{2p}{\mtx{X}_R}^{2p} \to 0$ as $R\to \infty$. This follows from the bound 
\[\pnorm{2p}{\mtx{X}_R}-\pnorm{2p}{\Expect \mtx{X}_R} \le \pnorm{2p}{\mtx{X}_R-\Expect \mtx{X}_R}\le \pnorm{2p}{\mtx{X}_R}+\pnorm{2p}{\Expect \mtx{X}_R},\]
and the fact that $\lim_{R\to \infty}\pnorm{2p}{\Expect \mtx{X}_R}=0$. To show this fact, first write
\begin{align*}\pnorm{2p}{\Expect \mtx{X}_R}&=\pnorm{2p}{-\Expect(-\mtx{X}+\mtx{X}_R)}=\pnorm{2p}{\Expect(-\mtx{X} \cdot \mathds{1}\{\pnorm{2p}{\mtx{X}} \ge R \})}\\
&\le \Expect(\pnorm{2p}{\mtx{X}} \cdot \mathds{1}\{\pnorm{2p}{\mtx{X}} \ge R \}),\end{align*}
and then apply the monotone convergence theorem.

To treat the right-hand side of~\eqref{eqn:poly-es-trunc},
observe that $\lim_{R \to \infty} \mtx{V}_R \to \mtx{V}$
almost surely because of the dominated convergence theorem
for conditional expectation.  Indeed,
\begin{align*}
\pnorm{p}{ \smash{\big( \mtx{X}_R - \mtx{X}_R^{(j)} \big)^2} \phantom{\big|}\!\! }
	&= \pnorm{2p}{ \smash{\mtx{X}_R - \mtx{X}_R^{(j)}}  \phantom{\big|}\!\! }^2 \\
	&\leq 2\pnorm{2p}{\mtx{X}_R}^{2} + 2\pnorm{2p}{\smash{\mtx{X}_R^{(j)}} \phantom{\big|}\!\!}^{2}
	\leq 2\pnorm{2p}{\mtx{X}}^{2} + 2\pnorm{2p}{\smash{\mtx{X}^{(j)}}  \phantom{\big|}\!\!}^{2}.
\end{align*}
The first relation follows from the identity $\pnorm{p}{\smash{\mtx{A}^2}} = \pnorm{2p}{\mtx{A}}^{2}$.
The right-hand side is integrable because $\mtx{X}^{(j)}$ has the same
distribution as $\mtx{X}$, and we can use Lyapunov's inequality to
increase the powers from two to $2p$.

We can apply the dominated convergence theorem again to see that
\begin{equation} \label{eqn:VR-lim}
\Expect \pnorm{p}{\mtx{V}_R}^{p} \to \Expect \pnorm{p}{\mtx{V}}^p.
\end{equation}
To see why, extend the argument in the last paragraph to reach
$$
\pnorm{p}{\mtx{V}_R}^p \leq (2n)^p \sum_{j=1}^n \Expect\left[
	\pnorm{2p}{\mtx{X}}^{2p} + \pnorm{2p}{\smash{\mtx{X}^{(j)}} \phantom{\big|}\!\!}^{2p} \bcondl Z \right]
$$
The right-hand side is integrable because of the tower property
and our assumption on the integrability of $\pnorm{2p}{\mtx{X}}$.

Take the limit of~\eqref{eqn:poly-es-trunc} as $R \to \infty$
using the expressions~\eqref{eqn:XR-lim} and~\eqref{eqn:VR-lim}.
This completes the proof of Theorem~\ref{thm:mxmoment}.

\section{Exponential Concentration Inequalities} \label{sec:concentration-bdd}

In this section, we develop an exponential moment bound for a kernel Stein pair.
This result shows that we can control the trace mgf in terms of the conditional
variance and the kernel conditional variance.

\begin{thm}[Exponential Moments for a Kernel Stein Pair]
\label{thm:concentration-unbdd}
Suppose that $(\mtx{X}, \mtx{X}') $ is a $\mtx{K}$-Stein pair,
and assume that $\norm{\mtx{X}}$ is bounded. 
For $\psi > 0$, define
\begin{equation} \label{eqn:rs-psi}
r(\psi) := 
	\frac{1}{\psi} \inf_{s > 0} \log \Expect \ntr \exp\left( \frac{\psi}{2} \big(s \, \mtx{V}_{\mtx{X}} + s^{-1} \, \mtx{V}^{\mtx{K}} \big)\right).
\end{equation}
When $\abs{\theta} < \sqrt{\psi}$,
$$
\log \Expect \ntr \econst^{\theta \mtx{X}}
	\leq \frac{\psi \, r(\psi)}{2} \log \left( \frac{1}{1-\theta^2/\psi} \right)
	\leq \frac{r(\psi) \, \theta^2}{2 (1 - \theta^2/\psi)}.
$$
The conditional variances $\mtx{V}_{\mtx{X}}$ and $\mtx{V}^{\mtx{K}}$ are defined
in~\eqref{eqn:conditional-variance} and \eqref{eqn:K-conditional-variance}.
\end{thm}

\noindent
The rest of this section is devoted to establishing this result.
The pattern of argument is similar with the proofs
of~\cite[Thm.~3.13]{Cha08:Concentration-Inequalities}
and~\cite[Thm.~5.1]{MackeyJoChFaTr12}, but we require
another nontrivial new matrix inequality.

Theorem~\ref{thm:concentration-unbdd} has a variety of
consequences.  In Section~\ref{sec:exp-es-proof},
we use it to derive the exponential Efron--Stein inequality, Theorem~\ref{thm:mxEfronStein}.
Additional applications of the result appear in Section~\ref{sec:beyond}.

\begin{rem}[Regularity Assumptions]
Theorem~\ref{thm:concentration-unbdd} instates a boundedness
assumption on $\mtx{X}$ to avoid some technical issues.
In fact, the result holds under weaker conditions.
\end{rem}

\subsection{Proof of Exponential Efron--Stein Inequality}
\label{sec:exp-es-proof}

Theorem~\ref{thm:concentration-unbdd} is the last major step toward
the matrix exponential Efron--Stein inequality, Theorem~\ref{thm:mxEfronStein}.
The proof is similar to the argument in Section~\ref{sec:poly-es-proof}
leading up to the polynomial Efron--Stein inequality so we proceed quickly.

Recall the setup from Section~\ref{sec:setup}.  We rely on the kernel
Stein pair $(\mtx{X}, \mtx{X}')$ that we constructed in Section~\ref{sec:my-kernel},
as well as the analysis from Section~\ref{eqn:my-evol}.
From Theorem~\ref{thm:concentration-unbdd} we obtain that for any $s>0$,
$$
\log \Expect \ntr \econst^{\theta \mtx{X}} \leq \frac{\theta^2/\psi}{2(1-\theta^2 / \psi)}
	\log \Expect \ntr \exp\left( \frac{s\psi}{2} \big( \mtx{V}_{\mtx{X}} + s^{-2} \,\mtx{V}^{\mtx{K}} \big) \right).
$$
Since $t \mapsto \econst^{s\psi t/2}$ is weakly increasing and convex on $\R_+$, by choosing $s$ as in \eqref{eq:sdef},
Lemma~\ref{lem:convex-conditional-variance-bound} implies that
\begin{align*}
\Expect \ntr \exp\left( \frac{s\psi}{2} \big( \mtx{V}_{\mtx{X}} + s^{-2} \,\mtx{V}^{\mtx{K}} \big) \right)
	&\leq \Expect \ntr \exp\left( \frac{s \psi}{2} \mtx{\Gamma}_0(Z) \right)\\
	&\leq \Expect \ntr \exp\left( 2 \frac{n \psi}{2} \mtx{\Gamma}_0(Z) \right)= \Expect \ntr \econst^{ 2\psi \mtx{V} }.
\end{align*}
The identity $\frac{n}{2} \mtx{\Gamma}_0(Z) = \mtx{V}$ was established in~\eqref{eq:my-Vdef}.
Combine the two displays, and make the change of variables $\psi \mapsto \psi/2$ to complete the proof of Theorem~\ref{thm:mxEfronStein}.

\subsection{The Exponential Mean Value Trace Inequality}

To establish Theorem~\ref{thm:concentration-unbdd},
we require another trace inequality.

\begin{lemma}[Exponential Mean Value Trace Inequality] \label{lem:emvti} 
For all matrices $\mtx{A}, \mtx{B}, \mtx{C} \in \Sym{d}$ and all $s > 0$ it holds that
\begin{align*}
\abs{\ntr \left[\mtx{C} (\econst^{\mtx{A}} - \econst^{\mtx{B}} )\right]} \leq  
	\frac{1}{4} \ntr\big[\big(s \, (\mtx{A}-\mtx{B})^2+ s^{-1} \, \mtx{C}^2\big)
	\big(\econst^{\mtx{A}}+\econst^{\mtx{B}}\big) \big] .
\end{align*}
\end{lemma}

\noindent
We defer the proof to Appendix~\ref{sec:emvti}.

\subsection{Some Properties of the Trace Mgf}

For the proof, we need to develop some basic facts about
the trace moment generating function.

\begin{lemma}[Properties of the Trace Mgf]
Assume that $\mtx{X} \in \Sym{d}$ is a centered
random matrix that is bounded in norm.
Define the normalized trace mgf $m(\theta) = \Expect \ntr \econst^{\theta \mtx{X}}$
for $\theta \in \R$.
Then
\begin{equation} \label{eqn:logm-nonneg}
\log m(\theta) \geq 0
\quad\text{and}\quad
\log m(0) = 0.
\end{equation}
The derivative of the trace mgf satisfies
\begin{equation} \label{eqn:mprime}
m'(\theta) = \Expect \ntr \big[ \mtx{X} \econst^{\theta \mtx{X}} \big]
\quad\text{and}\quad
m'(0) = 0.
\end{equation}
The trace mgf is a convex function; in particular
\begin{equation} \label{eqn:mprime-sign}
m'(\theta) \leq 0 \quad\text{for $\theta \leq 0$}
\quad\text{and}\quad
m'(\theta) \geq 0 \quad\text{for $\theta \geq 0$.}
\end{equation}
\end{lemma}

\begin{proof}
The result $m(0) = 0$ follows immediately from the definition of the trace mgf.
Since $\Expect \mtx{X} = \mtx{0}$,
$$
\log m(\theta) = \log \Expect \ntr \econst^{\theta \mtx{X}}
	\geq \log \ntr \econst^{\theta \Expect \mtx{X}}
	\geq 0.
$$
The first inequality is Jensen's, which depends on the fact~\eqref{eqn:trace-monotone}
that the trace exponential is a convex function.

Next, consider the derivative of the trace mgf.  For each $\theta \in \R$,
\begin{equation} \label{eqn:m-prime-v1}
m'(\theta)
	= \Expect \ntr \left[ \ddt{\theta} \, \econst^{\theta \mtx{X}} \right]
	= \Expect \ntr \big[ \mtx{X} \econst^{\theta \mtx{X}} \big],
\end{equation}
where the dominated convergence theorem and the boundedness of $\mtx{X}$
justify the exchange of expectation and derivative.
The claim $m'(0) = 0$ follows from~\eqref{eqn:m-prime-v1}
and the fact that $\Expect \mtx{X} = \mtx{0}$.

Similarly, the second derivative of the trace mgf satisfies
$$
m''(\theta) = \Expect \ntr \big[ \mtx{X}^2 \, \econst^{\theta\mtx{X}} \big]
	\geq 0
	\quad\text{for each $\theta \in \R$.}
$$
The inequality holds because $\mtx{X}^2$ and $\econst^{\theta \mtx{X}}$ are both positive semidefinite,
so the trace of their product must be nonnegative.  We discover that the trace
mgf is convex, which means that the derivative $m'$ is an increasing function.
\end{proof}

\subsection{Bounding the Derivative of the Trace Mgf} \label{sec:d-trace-mgf}

The first step in the proof of Theorem~\ref{thm:concentration-unbdd}
is to bound the trace mgf of the random matrix $\mtx{X}$ in terms of
the two conditional variance measures.

\begin{lemma}[The Derivative of the Trace Mgf] \label{lem:mgf-derivative}
Instate the notation and hypotheses of Theorem~\ref{thm:concentration-unbdd}.
Define the normalized trace mgf
$m(\theta) \defby \Expect \ntr \econst^{\theta \mtx{X}}$.
Then
\begin{align}
\abs{m'(\theta)}
	&\leq  \frac{1}{2} \abs{\theta} \cdot \inf_{s > 0}\ \Expect \ntr\big[\big(s\mtx{V}_{\mtx{X}} + s^{-1} \mtx{V}^{\mtx{K}} \big) \, \econst^{\theta \mtx{X}} \big]
	\quad\text{for all $\theta \in \R$.} \label{eqn:m-prime-Delta} 
\end{align}
\end{lemma}

\begin{proof}
Assume that the kernel Stein pair $(\mtx{X}, \mtx{X}')$ is constructed
from an auxiliary exchangable pair $(Z, Z')$.
By~\eqref{eqn:mprime}, the result holds trivially for $\theta = 0$,
so we may assume that $\theta \neq 0$.
The form of the derivative~\eqref{eqn:mprime}
is suitable for an application of the method of exchangeable pairs, Lemma~\ref{lem:exchange}.
Since $\mtx{X}$ is bounded, the regularity condition~\eqref{eqn:regularity-mep} is satisfied, and we obtain
\begin{equation} \label{eqn:m-prime-v2}
m'(\theta) = \frac{1}{2}
	\Expect \ntr \big[ \mtx{K}(Z,Z') \big(\econst^{\theta \mtx{X}} - \econst^{\theta \mtx{X}'} \big) \big].
\end{equation}
The exponential mean value trace inequality, Lemma~\ref{lem:emvti}, implies that
\begin{align*}
|m'(\theta)|
&\leq \frac{1}{8} \cdot \inf_{s > 0}\ \Expect \ntr \big[
	\big(s \, (\theta\mtx{X}-\theta\mtx{X}')^2+s^{-1}\mtx{K}(Z,Z')^2 \big) \cdot \big(\econst^{\theta \mtx{X}} + \econst^{\theta \mtx{X}'} \big) \big] \\
&= \frac{1}{4} \cdot \inf_{s > 0}\ \Expect \ntr \big[
	\big(s \, (\theta\mtx{X}-\theta\mtx{X}')^2+s^{-1}\mtx{K}(Z,Z')^2 \big) \cdot \econst^{\theta \mtx{X}} \big] \\
&= \frac{1}{4} \abs{\theta} \cdot \inf_{t > 0}\ \Expect \ntr \big[
	\big(t \, (\mtx{X}-\mtx{X}')^2+t^{-1} \mtx{K}(Z,Z')^2 \big) \cdot \econst^{\theta \mtx{X}} \big] \\
&=  \frac{1}{2} \abs{\theta} \cdot \inf_{t > 0}\ \Expect \ntr \left[
	\frac{t}{2} \Expect\big[ (\mtx{X} - \mtx{X}')^2 \bcondl Z\big]
	\cdot \econst^{\theta \mtx{X}}  
	+ \frac{1}{2t} \Expect\big[ \mtx{K}(Z,Z')^2\bcondl Z\big] \cdot \econst^{\theta \mtx{X}}  \right]. 
\end{align*}
The first equality follows from the exchangeability of $(\mtx{X},\mtx{X}')$;
the second follows from the change of variables $s = \abs{\theta}^{-1} t$; 
and the final one depends on the pull-through property of conditional expectation.
We reach the result \eqref{eqn:m-prime-Delta} by introducing the definitions
\eqref{eqn:conditional-variance} and \eqref{eqn:K-conditional-variance} of the conditional variance and the kernel conditional variance.
\end{proof}

\subsection{Decoupling via an Entropy Inequality}

The next step in the proof uses an entropy inequality to separate
the conditional variances in~\eqref{eqn:m-prime-Delta} from the
matrix exponential.

\begin{fact}[Young's Inequality for Matrix Entropy] \label{fact:young}
Let $\mtx{U}$ be a random matrix in $\Sym{d}$ that is
bounded in norm, and suppose that $\mtx{W}$
is a random matrix in $\Sym{d}_+$ that is subject to the
normalization $\Expect \ntr \mtx{W} = 1$.  Then
$$
\Expect \ntr(\mtx{UW})
	\leq \log \Expect \ntr \econst^{\mtx{U}}
	+ \Expect \ntr [ \mtx{W} \log \mtx{W} ].
$$
\end{fact}

\noindent
This fact appears as~\cite[Prop.~A.3]{MackeyJoChFaTr12};
see also~\cite[Thm.~2.13]{Car10:Trace-Inequalities}.

\subsection{A Differential Inequality}

To continue the argument, we fix a parameter $\psi > 0$.
Rewrite~\eqref{eqn:m-prime-Delta} as
$$
\abs{m'(\theta)} \leq \frac{\abs{\theta} m(\theta)}{\psi} \inf_{s > 0}
	\Expect\ntr \left[ \left( \frac{\psi}{2} \big( s \mtx{V}_{\mtx{X}} + s^{-1} \mtx{V}^{\mtx{K}}\big) \right) \cdot
	\frac{\econst^{\theta \mtx{X}}}{m(\theta)} \right]
$$
Invoke Fact~\ref{fact:young} to obtain
\begin{multline*}
\abs{m'(\theta)} \leq \frac{\abs{\theta} m(\theta)}{\psi} \bigg( \inf_{s > 0}
	\log \Expect \ntr \exp\left( \frac{\psi}{2} \big( s \mtx{V}_{\mtx{X}} + s^{-1} \mtx{V}^{\mtx{K}} \big) \right) \\
	+ \Expect \ntr \left[ \frac{\econst^{\theta \mtx{X}}}{m(\theta)}
	\log \frac{\econst^{\theta{\mtx{X}}}}{m(\theta)} \right] \bigg).
\end{multline*}
In view of~\eqref{eqn:logm-nonneg},
$$
\log \frac{\econst^{\theta \mtx{X}}}{m(\theta)}
	= \theta \mtx{X} - \log m(\theta) \cdot \Id
	\psdle \theta \mtx{X}.
$$
Identify the function $r(\psi)$ defined in~\eqref{eqn:rs-psi}
and the derivative of the trace mgf to reach
\begin{equation} \label{eqn:unbdd-diff-ineq}
\abs{ m'(\theta) }
	\leq \abs{\theta} m(\theta) \, r(\psi) + \frac{\theta \abs{\theta}}{\psi} \cdot m'(\theta).
\end{equation}
This inequality is valid for all $\psi > 0$, and all $\theta \in \R$.

\subsection{Solving the Differential Inequality}

We begin with the case where $\theta \geq 0$.  The result~\eqref{eqn:mprime-sign}
shows that $m'(\phi) \geq 0$ for $\phi \in [0, \theta]$.  Therefore,
the differential inequality~\eqref{eqn:unbdd-diff-ineq} reads
$$
m'(\phi) \leq \phi \, m(\phi) \, r(\psi) + (\phi^2/\psi) \, m'(\phi)
\quad\text{for $\phi \in [0, \theta]$.}
$$
Rearrange this expression to isolate the log-derivative $m'(\phi)/m(\phi)$:
$$
\frac{\diff{}}{\diff{\phi}} \log m(\phi)
	\leq \frac{r(\psi) \, \phi}{1 - \phi^2/\psi}
	\quad\text{when $0 \leq \phi \leq \theta < \sqrt{\psi}$.}
$$
Recall the fact~\eqref{eqn:logm-nonneg} that $\log m(0) = 0$, and integrate to obtain
$$
\log m(\theta) = \int_0^\theta \frac{\diff{}}{\diff{\phi}} \log m(\phi) \idiff{\phi}
	\leq \int_0^\theta \frac{r(\psi) \, \phi}{1 - \phi^2/\psi} \idiff{\phi}
	= \frac{\psi\, r(\psi)}{2} \log \left( \frac{1}{1 - \theta^2/\psi} \right)
$$
when $0 \leq \theta < \sqrt{\psi}$.  Making an additional approximation, we find that
$$
\log m(\theta)
	\leq \int_0^\theta \frac{r(\psi) \, \phi}{1 - \theta^2/\psi} \idiff{\phi}
	= \frac{r(\psi) \, \theta^2}{2(1 - \theta^2 /\psi)}
$$
for the same parameter range.

Finally, we treat the case where $\theta \leq 0$.
The result~\eqref{eqn:mprime-sign} shows that $m'(\phi) \leq 0$ for $\phi \in [\theta, 0]$,
so the differential inequality~\eqref{eqn:unbdd-diff-ineq} becomes
$$
m'(\phi) \geq \phi \, m(\phi) \, r(\psi) + (\phi^2/\psi) \, m'(\phi)
\quad\text{for $\phi \in [\theta, 0]$.}
$$
The rest of the argument parallels the situation where $\theta$ is positive.

\section{Complements} \label{sec:beyond}

The tools in this paper are applicable in a wide variety of settings.  To indicate what might be possible, we briefly present two additional concentration results for random matrices arising as functions of dependent random variables.  We also indicate some prospects for future research.

\subsection{Matrix Bounded Differences without Independence} \label{sec:dobrushin}

A key strength of the method of exchangeable pairs is the fact that it also applies
to random matrices that are built from weakly dependent random variables.  This
section describes an extension of Corollary~\ref{cor:bound-diff} that holds even
when the input variables exhibit some interactions.

To quantify the amount of dependency among the variables, we use
a Dobrushin interdependence matrix~\citep{dobrushin1970prescribing}.
This concept involves a certain amount of auxiliary notation.
Given a vector $\vct{x} = (x_1, \dots, x_n)$,
we write \[\vct{x}_{-i}= (x_1, \dots x_{i-1}, x_{i+1}, \dots, x_n)\]
for the vector with its $i$th component deleted.
Let $\Zvec = (Z_1, \dots, Z_n)$ be a vector of random variables
taking values in a Polish space $\metricspace$ with sigma algebra
$\mathcal{F}$.  The symbol $\mu_i(\cdot \bcondl \Zvec_{-i})$ refers to
the distribution of $\Zvec_i$ conditional on the random vector
$\Zvec_{-i}$.  We also require the total variation distance $\tv$
between probability measures $\mu$ and $\nu$ on $(\metricspace, \mathcal{F})$:
\begin{align} \label{eqn:tv}
\tv(\nu,\mu) \defby \sup_{A\in\mathcal{F}} \abs{\nu(A) - \mu(A)}.
\end{align}
With this foundation in place, we can state the definition.

\begin{defn}[Dobrushin Interdependence Matrix] \label{def:dobrushin}
Let $\Zvec=(Z_1, \ldots,$ $Z_n)$ be a random vector taking values in a Polish space $\metricspace$.
Let $\mtx{D} \in \R^{n \times n}$ be a matrix with a zero diagonal that satisfies
the condition
\begin{align} \label{eqn:dobrushin}
\tv\big(\mu_i(\cdot\bcondl  \vct{x}_{-i}),\mu_i(\cdot\bcondl  \vct{y}_{-i}) \big)
\leq \sum\nolimits_{j=1}^n D_{ij}\II[x_j\ne y_j]
\end{align}
for each index $i$ and for all vectors $\vct{x}, \vct{y} \in \metricspace$.  Then $\mtx{D}$ is called a \term{Dobrushin interdependence matrix} for the random vector $\Zvec$.
\end{defn}

The kernel coupling method extends readily to the setting of weak dependence.
We obtain a new matrix bounded differences inequality, which is a significant
extension of Corollary~\ref{cor:bound-diff}.  This statement can be viewed as
a matrix version of Chatterjee's result~\cite[Thm.~4.3]{Cha08:Concentration-Inequalities}.

\begin{cor}[Dobrushin Matrix Bounded Differences] \label{cor:dob-bound-diff}
	Suppose that $\Zvec \defby (Z_1, \dots, Z_n)$ in a Polish space $\metricspace$ is a vector of dependent random variables with a Dobrushin interdependence matrix $\mtx{D}$
with the property that
\begin{align}\label{eqn:dobrushin-constraint}
\max\big\{ \indnorm{1}{\mtx{D}}, \ \indnorm{\infty}{\mtx{D}} \big\} < 1.
\end{align} 
Let $\mtx{H} : \metricspace \to \Sym{d}$ be a measurable function,
and let $(\mtx{A}_1, \dots, \mtx{A}_n)$
be a deterministic sequence of Hermitian matrices that satisfy
$$
(\mtx{H}(z_1,\dots, z_n) - \mtx{H}(z_1,\dots,z_j',\dots,z_n))^2 \psdle \mtx{A}_j^2
$$
where $z_k, z_k'$ range over the possible values of $Z_k$ for each $k$. Compute the 
boundedness and dependence parameters
\begin{equation*}
\sigma^2\defby \norm{ \sum\nolimits\nolimits_{j=1}^n \mtx{A}_j^2 }
\qtext{and}
b \defby \left[ 1-\frac{1}{2} \big(\indnorm{1}{\mtx{D}}+\indnorm{\infty}{\mtx{D}} \big) \right]^{-1}.
\end{equation*}
Then, for all $t\ge 0$,
\begin{equation*}
\Prob{\lambda_{\max}\left(\mtx{H}(Z)-\Expect\mtx{H}(Z)\right)\ge t}\le d\cdot \econst^{-t^2/(b\sigma^2)}.
\end{equation*}
Furthermore,
	$$
	\Expect \lambda_{\max}\left( \mtx{H}(Z) - \Expect \mtx{H}(Z) \right) \leq \sigma \sqrt{b\log d}.
	$$ 
\end{cor}

Observe that the bounds here are a factor of $b$ worse than the independent case outlined in Corollary~\ref{cor:bound-diff}.  The proof is similar to the proof in the scalar case in \citep{Cha08:Concentration-Inequalities}.  We refer
the reader to our earlier report~\citep{PMT13:Deriving-Matrix} for details.

\subsection{Matrix-Valued Functions of Haar Random Elements}

This section describes a concentration result for a matrix-valued function of a random element drawn uniformly from a compact group.  This corollary can be viewed as a matrix extension of~\cite[Thm.~4.6]{Cha08:Concentration-Inequalities}.

\begin{cor}[Concentration for Hermitian Functions of Haar Measures] \label{cor:concentration-haar}
	Let $Z\sim \mu$ be Haar distributed on a compact topological group $G$, and 
	let $\mtx{\Psi} : G \to \Sym{d}$ be a measurable function satisfying $\Expect{\mtx{\Psi}(Z)} = \mtx{0}$.
	Let $Y, Y_1, Y_2, \dots$ be i.i.d.~random variables in $G$ satisfying 
	\begin{align} \label{eqn:Y-property}
	Y \sim Y^{-1} \qtext{and} zYz^{-1} \sim Y \qtext{for all } z \in G.
	\end{align}
	Assume
	$$
	\norm{\mtx{\Psi}(z)} \leq R \qtext{for all} z\in G,
	$$
	and 
	$$
	S^2 = \sup_{g \in G}\norm{\Expect\big[(\mtx{\Psi}(g) - \mtx{\Psi}(Yg))^2 \big]} < \infty.
	$$
	Compute the boundedness parameter
	$$
	\sigma^2 \defby \frac{S^2}{2}\sum\nolimits_{i=0}^\infty \min\big\{1,\ 4RS^{-1} \tv(\mu_i, \mu) \big\}
	$$
	where $\mu_i$ is the distribution of the product ${Y}_i\cdots {Y}_1$.
	Then, for all $t\geq 0$, 
	$$
	\Prob{ \lambda_{\max}\left( \mtx{\Psi}(Z) \right) \geq t }
		\leq d \cdot \econst^{-t^2/(2\sigma^2)}.
	$$
	Furthermore,
	$$
	\Expect \lambda_{\max}\left( \mtx{\Psi}(Z) \right) \leq \sigma \sqrt{2\log d}.
	$$ 
\end{cor}

\cororef{concentration-haar} relates the concentration of Hermitian functions to the convergence of random walks on a group.
In particular, \cororef{concentration-haar} can be used to study matrices constructed from random permutations or random unitary matrices. 
The proof is similar to the proof of the scalar result; see our earlier report~\citep{PMT13:Deriving-Matrix} for details.

\subsection{Conjectures and Consequences}
\label{sec:conjectures}

We conjecture that the following trace inequalities hold.
\begin{conjecture}[Signed Mean Value Trace Inequalities] \label{conj:mvti}
For all matrices $\mtx{A}, \mtx{B}, \mtx{C} \in \Sym{d}$, all positive integers $q$,  and any $s > 0$ it holds that
\begin{align*}
\trace \big[\mtx{C} (\econst^{\mtx{A}} - \econst^{\mtx{B}} )\big]
	\leq \frac{1}{2} \trace &\big[(s \, (\mtx{A} -\mtx{B})_+^2+s^{-1} \, \mtx{C}_+^2) \, \econst^{\mtx{A}} \\
	& + (s \,(\mtx{A}-\mtx{B})_-^2+s^{-1}\, \mtx{C}_-^2) \, \econst^{\mtx{B}}) \big] .
\end{align*}
and
\begin{align*}
\trace \big[\mtx{C}(\mtx{A}^{q} - \mtx{B}^{q} )\big]
	\leq  \frac{q}{2} \trace &\big[(s \, (\mtx{A}-\mtx{B})_+^2+s^{-1}\,\mtx{C}_+^2)\abs{\mtx{A}}^{q-1} \\
	&+ (s \, (\mtx{A}-\mtx{B})_-^2+s^{-1}\,\mtx{C}_-^2)\abs{\mtx{B}}^{q-1}) \big].
\end{align*}
\end{conjecture}

\noindent
This statement involves the standard matrix functions that lift the scalar functions $(a)_+ := \max\{ a, 0 \}$ and $(a)_- := \max\{-a, 0\}$.  Extensive simulations with random matrices suggest that Conjecture~\ref{conj:mvti} holds, but we did not find a proof.

These inequalities would imply one-sided matrix versions of the exponential Efron--Stein and moment bounds, similar to those formulated for the scalar setting in \citep{BoLuMa2003} and \citep{BoLuMa2005}.
In the scalar case, Conjecture \ref{conj:mvti} is valid, so it is possible to obtain the results of \citep{BoLuMa2003} and \citep{BoLuMa2005} by the exchangeable pair method.

\appendix

\section{Operator Inequalities} \label{sec:operator-ineq}

Our main results rely on some basic inequalities from operator theory.
We are not aware of good references for this material, so we have
included short proofs.

\subsection{Young's Inequality for Commuting Operators}

In the scalar setting, Young's inequality provides an additive
bound for the product of two numbers.  More precisely,
for indices $p, q \in (1, \infty)$ that satisfy the
conjugacy relation $p^{-1} + q^{-1} = 1$, we have
\begin{equation} \label{eqn:young-scalar}
ab \leq \frac{1}{p} \abs{a}^p + \frac{1}{q} \abs{b}^q
\quad\text{for all $a, b \in \R$.}
\end{equation}
The same result has a natural extension for commuting operators.

\begin{lemma}[Young's Inequality for Commuting Operators]
\label{lem:young-commute}
Suppose that $\mathcal{A}$ and $\mathcal{B}$ are self-adjoint linear maps on
the Hilbert space $\M^d$ that commute with each other.
Let $p, q \in (1, \infty)$ satisfy the conjugacy relation
$p^{-1} + q^{-1} = 1$.  Then
$$
\mathcal{A}\mathcal{B}
	\psdle \frac{1}{p} \abs{\mathcal{A}}^p + \frac{1}{q} \abs{\mathcal{B}}^q.
$$
\end{lemma}

\begin{proof}
Since $\mathcal{A}$ and $\mathcal{B}$ commute,
there exists a unitary operator $\mathcal{U}$
and diagonal operators $\mathcal{D}$ and $\mathcal{M}$ for which
$\mathcal{A} = \mathcal{U}\mathcal{D}\mathcal{U}^*$ and
$\mathcal{B} = \mathcal{U}\mathcal{M}\mathcal{U}^*$.
Young's inequality~\eqref{eqn:young-scalar} for scalars
immediately implies that
$$
\mathcal{D}\mathcal{M} \psdle
	\frac{1}{p}\abs{\mathcal{D}}^p + \frac{1}{q}\abs{\mathcal{M}}^q.
$$
Conjugating both sides of this inequality by $\mathcal{U}$, we obtain
$$
\mathcal{A}\mathcal{B}
	= \mathcal{U}(\mathcal{D}\mathcal{M})\mathcal{U}^* 
	\psdle \frac{1}{p}\mathcal{U} \abs{\mathcal{D}}^p \mathcal{U}^*
	+ \frac{1}{q}\mathcal{U}\abs{\mathcal{M}}^q\mathcal{U}^*
	= \frac{1}{p}\abs{\mathcal{A}}^p + \frac{1}{q}\abs{\mathcal{B}}^q.
$$
The last identity follows from the definition of a standard function of an
operator.
\end{proof}

\subsection{An Operator Version of Cauchy--Schwarz}

We also need a simple version of the Cauchy--Schwarz inequality
for operators.  The proof follows a classical argument, but it
also involves an operator decomposition.

\begin{lemma}[Operator Cauchy--Schwarz] \label{lem:operator-cs}
Let $\mathcal{A}$ be a self-adjoint linear operator on the Hilbert space $\M^d$,
and let $\mtx{M}$ and $\mtx{N}$ be matrices in $\M^d$.  Then
$$
\absip{ \mtx{M} }{ \mathcal{A}(\mtx{N}) }
	\leq \big[ \ip{ \mtx{M} }{ \abs{\mathcal{A}}(\mtx{M}) } \cdot
	\ip{ \mtx{N} }{ \abs{\mathcal{A}}(\mtx{N}) } \big]^{1/2}.
$$
The inner product symbol refers to the trace, or Frobenius, inner product.
\end{lemma}

\begin{proof}
Consider the Jordan decomposition $\mathcal{A} = \mathcal{A}_{+} - \mathcal{A}_{-}$,
where $\mathcal{A}_+$ and $\mathcal{A}_{-}$ are both positive semidefinite.  For all $s > 0$,
\begin{align*}
0 &\leq \ip{(s \mtx{M} - s^{-1} \mtx{N}) }{ \mathcal{A}_+(s \mtx{M} - s^{-1}\mtx{N}) } \\
	&= s^2 \ip{ \mtx{M} }{\mathcal{A}_+(\mtx{M}) }
	+ s^{-2} \ip{ \mtx{N} }{ \mathcal{A}_+(\mtx{N}) }
	- 2 \ip{ \mtx{M} }{ \mathcal{A}_+( \mtx{N} ) }.
\end{align*}
Likewise,
\begin{align*}
0 &\leq \ip{(s \mtx{M} + s^{-1} \mtx{N}) }{ \mathcal{A}_{-}(s \mtx{M} + s^{-1}\mtx{N}) } \\
	&= s^2 \ip{ \mtx{M} }{\mathcal{A}_{-}(\mtx{M}) }
	+ s^{-2} \ip{ \mtx{N} }{ \mathcal{A}_{-}(\mtx{N}) }
	+ 2 \ip{ \mtx{M} }{ \mathcal{A}_{-}( \mtx{N} ) }. 
\end{align*}
Add the latter two inequalities and rearrange the terms to obtain
$$
2 \ip{ \mtx{M} }{ \mathcal{A}(\mtx{N}) }
	\leq s^2 \ip{ \mtx{M} }{ \abs{\mathcal{A}}(\mtx{M}) }
	+ s^{-2} \ip{ \mtx{N} }{ \abs{\mathcal{A}}(\mtx{N}) },
$$
where we have used the relation $\abs{\mathcal{A}} = \mathcal{A}_+ + \mathcal{A}_-$.  Take the infimum of the right-hand side over $s > 0$ to reach
\begin{equation} \label{eqn:cs-halfway}
\ip{ \mtx{M} }{ \mathcal{A}( \mtx{N} ) }
	\leq \big[ \ip{ \mtx{M} }{\abs{\mathcal{A}}(\mtx{M}) } \cdot 
	\ip{ \mtx{M} }{ \abs{\mathcal{A}}( \mtx{N} ) } \big]^{1/2}.
\end{equation}
Repeat the same argument, interchanging the roles of the matrices
$s \mtx{M} - s^{-1} \mtx{N}$ and $s \mtx{M} + s^{-1} \mtx{N}$.
We conclude that~\eqref{eqn:cs-halfway} also holds with an absolute value on the
left-hand side.  This observation completes the proof.
\end{proof}

\section{The Polynomial Mean Value Trace Inequality}
\label{sec:proof-pmvti}

The critical new ingredient in Theorem~\ref{thm:bdg-inequality}
is the polynomial mean value trace inequality, Lemma~\ref{lem:pmvti}.
Let us proceed with a proof of this result.

\begin{proof}[Proof of Lemma~\ref{lem:pmvti}] 
First, we need to develop another representation for the trace
quantity that we are analyzing.  Assume that $\mtx{A}, \mtx{B}, \mtx{C} \in \Sym{d}$.  A direct calculation shows that
\begin{align*}
\mtx{A}^{q} - \mtx{B}^{q} 
	= \sum\nolimits_{k=0}^{q-1}\mtx{A}^{k} (\mtx{A} - \mtx{B}) \mtx{B}^{q-1-k}.
\end{align*}
As a consequence,
\begin{equation} \label{eqn:poly-trace-sum}
\trace \left[\mtx{C} (\mtx{A}^{q} - \mtx{B}^{q} )\right] 
	= \sum\nolimits_{k=0}^{q-1} \ip{ \mtx{C} }{ \mtx{A}^{k} (\mtx{A} - \mtx{B}) \mtx{B}^{q-1-k} }.
\end{equation}
To bound the right-hand side of~\eqref{eqn:poly-trace-sum}, we require an approriate mean inequality.

To that end, we define some self-adjoint operators on $\M^d$:
$$
\mathcal{A}_k(\mtx{M}) \defby \mtx{A}^k \mtx{M}
\qtext{and}
\mathcal{B}_{k}(\mtx{M}) \defby  \mtx{M} \mtx{B}^{k}
\quad\text{for each $k = 0, 1, 2, \dots, q-1$.} 
$$
The absolute values of these operators satisfy
$$
\abs{\mathcal{A}_k}(\mtx{M}) = \abs{\mtx{A}}^k \mtx{M}
\qtext{and}
\abs{\mathcal{B}_{k}}(\mtx{M}) =  \mtx{M} \abs{\mtx{B}}^{k}
\quad\text{for each $k = 0, 1, 2, \dots, q-1$.} 
$$
Note that $\abs{\mathcal{A}_k}$ and $\abs{\mathcal{B}_{q-k-1}}$ commute with each other
for each $k$.  Therefore, Young's inequality for commuting operators, Lemma~\ref{lem:young-commute}, yields the bound
\begin{align} 
\abs{\mathcal{A}_k\mathcal{B}_{q-k-1}}
	= \abs{\mathcal{A}_k} \abs{\mathcal{B}_{q-k-1}}
	&\psdle \frac{k}{q-1}\abs{\mathcal{A}_k}^{(q-1)/k}
	+ \frac{q-k-1}{q-1} \abs{\mathcal{B}_{q-k-1}}^{(q-1)/(q-k-1)} \notag \\
	&= \frac{k}{q-1} \abs{\mathcal{A}_1}^{q-1} + \frac{q-k-1}{q-1}\abs{\mathcal{B}_1}^{q-1}. \label{eqn:poly-mean-ineq}
\end{align}
Summing over $k$, we discover that
\begin{align}\label{eqn:young-bound}
\sum\nolimits_{k=0}^{q-1}\abs{\mathcal{A}_k\mathcal{B}_{q-k-1} }
	&\psdle \frac{q}{2} \abs{\mathcal{A}_1}^{q-1} + \frac{q}{2}\abs{\mathcal{B}_1}^{q-1}.
\end{align}
This is the mean inequality that we require.

To apply this result, we need to rewrite~\eqref{eqn:poly-trace-sum} using the operators $\mathcal{A}_k$ and $\mathcal{A}_{q-k-1}$.  It holds that
\begin{multline} \label{eqn:poly-sums}
\trace \left[\mtx{C} (\mtx{A}^{q} - \mtx{B}^{q} )\right]
	= \sum_{k=0}^{q-1} \ip{ \mtx{C} }{ (\mathcal{A}_k \mathcal{B}_{q-k-1})(\mtx{A}-\mtx{B}) } \\
	\leq \left[ \sum_{k=0}^{q-1} \ip{ \mtx{C} }{ \abs{\mathcal{A}_k \mathcal{B}_{q-k-1}}(\mtx{C}) } \cdot \sum_{k=0}^{q-1} \ip{ \mtx{A} - \mtx{B} }{  \abs{\mathcal{A}_k \mathcal{B}_{q-k-1}}(\mtx{A}-\mtx{B}) } \right]^{1/2}. 
\end{multline}
The second relation follows from the operator Cauchy--Schwarz inequality, Lem\-ma~\ref{lem:operator-cs}, and the usual Cauchy--Schwarz inequality for the sum.

It remains to bound to two sums on the right-hand side of~\eqref{eqn:poly-sums}. 
The mean inequality~\eqref{eqn:poly-mean-ineq} ensures that
\begin{multline} \label{eqn:poly-sum-1}
\sum_{k=0}^{q-1} \ip{ \mtx{C} }{ \abs{\mathcal{A}_k \mathcal{B}_{q-k-1}}(\mtx{C}) }
	\leq \frac{q}{2} \ip{ \mtx{C} }{ \big(\abs{\mathcal{A}_1}^{q-1} + \abs{\mathcal{B}_1}^{q-1} \big)(\mtx{C}) } \\
	= \frac{q}{2} \ip{ \mtx{C} }{ \abs{\mtx{A}}^{q-1} \mtx{C} + \mtx{C} \abs{\mtx{B}}^{q-1} }
	= \frac{q}{2} \trace\big[ \mtx{C}^2 \big( \abs{\mtx{A}}^{q-1} + \abs{\mtx{B}}^{q-1} \big) \big].
\end{multline}
Likewise,
\begin{equation} \label{eqn:poly-sum-2}
\sum_{k=0}^{q-1} \ip{ \mtx{A} - \mtx{B} }{ \abs{\mathcal{A}_k \mathcal{B}_{q-k-1}}(\mtx{A}-\mtx{B}) }
	\leq \frac{q}{2} \trace\big[ (\mtx{A}-\mtx{B})^2 \big( \abs{\mtx{A}}^{q-1} + \abs{\mtx{B}}^{q-1} \big) \big].
\end{equation}
Introduce the two inequalities~\eqref{eqn:poly-sum-1} and~\eqref{eqn:poly-sum-2} into~\eqref{eqn:poly-sums} to reach
\begin{align*}
&\trace \left[\mtx{C} (\mtx{A}^{q} - \mtx{B}^{q} )\right]
	\\
	&\quad\leq \frac{q}{2} \bigg( \trace\big[ \mtx{C}^2 \big( \abs{\mtx{A}}^{q-1} + \abs{\mtx{B}}^{q-1} \big) \big] \cdot
	\trace\big[ (\mtx{A}-\mtx{B})^2 \big( \abs{\mtx{A}}^{q-1} + \abs{\mtx{B}}^{q-1} \big) \big] \bigg)^{1/2}.
\end{align*}
The result follows when we apply the numerical inequality between the geometric mean and the arithmetic mean.
\end{proof}

\section{The Exponential Mean Value Trace Inequality} \label{sec:emvti}

Finally, we establish the trace inequality stated in Lemma~\ref{lem:emvti}.
See the manuscript \citep{Anote} for an alternative proof.

\begin{proof}[Proof of Lemma~\ref{lem:emvti}] 
To begin, we develop an alternative expression for the trace
quantity that we need to bound.  Observe that 
$$
\frac{\diff{}}{\diff{\tau}} \econst^{\tau \mtx{A}} \econst^{(1-\tau)\mtx{B}}
	= \econst^{\tau \mtx{A}}(\mtx{A} - \mtx{B}) \econst^{(1-\tau)\mtx{B}}.
$$
The Fundamental Theorem of Calculus delivers the identity
$$
\econst^{\mtx{A}} - \econst^{\mtx{B}}
	= \int_0^1 \frac{\diff{}}{\diff{\tau}}
	\econst^{\tau \mtx{A}} \econst^{(1-\tau)\mtx{B}} \idiff{\tau}
	= \int_0^1 \econst^{\tau \mtx{A}}(\mtx{A} - \mtx{B}) \econst^{(1-\tau)\mtx{B}}
	\idiff{\tau}.
$$
Therefore, using the definition of the trace inner product, we reach
\begin{equation} \label{eqn:emvti-integral}
\trace \big[ \mtx{C} \big( \econst^{\mtx{A}} - \econst^{\mtx{B}} \big) \big]
	= \int_0^1 \ip{ \mtx{C} }{ \econst^{\tau \mtx{A}}(\mtx{A} - \mtx{B}) \econst^{(1-\tau)\mtx{B}} } \diff{\tau}.
\end{equation}
We can bound the right-hand side by developing an appropriate
matrix version of the inequality between the logarithmic
mean and the arithmetic mean.

Let us define two families of positive-definite
operators on the Hilbert space $\M^{d}$:
$$
\mathcal{A}_{\tau}(\mtx{M}) = \econst^{\tau \mtx{A}} \mtx{M}
\quad\text{and}\quad
\mathcal{B}_{1-\tau}(\mtx{M}) = \mtx{M} \econst^{(1-\tau) \mtx{B}}
\quad\text{for each $\tau \in [0,1]$.}
$$
In other words, $\mathcal{A}_{\tau}$ is a left-multiplication operator,
and $\mathcal{B}_{1-\tau}$ is a right-multi\-plication operator.  It follows
immediately that $\mathcal{A}_{\tau}$ and $\mathcal{B}_{1-\tau}$ commute
for each $\tau \in [0,1]$.
Young's inequality for commuting operators, Lemma~\ref{lem:young-commute},
implies that
\begin{equation*} 
\mathcal{A}_\tau\mathcal{B}_{1-\tau} 
	\psdle \tau \cdot \abs{\mathcal{A}_\tau}^{1/\tau} + (1-\tau) \cdot \abs{\mathcal{B}_{1-\tau}}^{1/(1-\tau)}
	=  \tau \cdot \abs{\mathcal{A}_1} + (1-\tau) \cdot \abs{\mathcal{B}_{1}}.
\end{equation*}
Integrating over $\tau$, we discover that
\begin{align}\label{eqn:young-exp-bound}
\int\nolimits_{0}^1\mathcal{A}_\tau\mathcal{B}_{1-\tau} d\tau
	&\psdle \frac{1}{2} (|\mathcal{A}_1| + |\mathcal{B}_{1}|)
	= \frac{1}{2} (\mathcal{A}_1 + \mathcal{B}_{1}).
\end{align}
This is our matrix extension of the logarithmic--arithmetic mean inequality.

To relate this result to the problem at hand, we rewrite
the expression~\eqref{eqn:emvti-integral} using the operators
$\mathcal{A}_{\tau}$ and $\mathcal{B}_{1-\tau}$.  Indeed,
\begin{multline}
\trace \big[ \mtx{C} \big( \econst^{\mtx{A}} - \econst^{\mtx{B}} \big) \big]
	= \int_0^1 \ip{ \mtx{C} }{ (\mathcal{A}_{\tau} \mathcal{B}_{1-\tau} )(\mtx{A} - \mtx{B}) } \diff{\tau} \\
	\leq \left[ \int_0^1 \ip{ \mtx{C} }{ (\mathcal{A}_{\tau} \mathcal{B}_{1-\tau} )(\mtx{C})} \diff{\tau} \cdot
	\int_0^1 \ip{ \mtx{A} - \mtx{B} }{ (\mathcal{A}_{\tau} \mathcal{B}_{1-\tau} )(\mtx{A} - \mtx{B})} \diff{\tau} \right]^{1/2}.
	\label{eqn:exp-bd-integrals}
\end{multline}
The second identity follows from the definition of the trace inner product.
The last relation follows from the operator Cauchy--Schwarz inequality, Lemma~\ref{lem:operator-cs}, and the usual Cauchy--Schwarz inequality for the integral.

It remains to bound the two integrals in~\eqref{eqn:exp-bd-integrals}.  These estimates are an immediate consequence of~\eqref{eqn:young-exp-bound}.  First,
\begin{multline} \label{eqn:exp-integral-1}
\int_0^1 \ip{ \mtx{C} }{ (\mathcal{A}_{\tau} \mathcal{B}_{1-\tau} )(\mtx{C})} \idiff{\tau}
	\leq \frac{1}{2} \ip{ \mtx{C} }{ (\mathcal{A}_1 + \mathcal{B}_1)(\mtx{C}) } \\
	= \frac{1}{2} \ip{ \mtx{C} }{ \econst^{\mtx{A}} \mtx{C} + \mtx{C} \econst^{\mtx{B}} }
	= \frac{1}{2} \trace\big[ \mtx{C}^2 \big( \econst^{\mtx{A}} + \econst^{\mtx{B}} \big) \big]. 
\end{multline}
The last two relations follow from the definitions of the operators $\mathcal{A}_1$ and $\mathcal{B}_1$, the definition of the trace inner product, and the cyclicity of the trace.  Likewise,
\begin{equation} \label{eqn:exp-integral-2}
\int_0^1 \ip{ \mtx{A} - \mtx{B} }{ (\mathcal{A}_{\tau} \mathcal{B}_{1-\tau} )(\mtx{A} - \mtx{B})} \diff{\tau}
= \frac{1}{2} \trace\big[ (\mtx{A}-\mtx{B})^2 \big( \econst^{\mtx{A}} + \econst^{\mtx{B}} \big) \big].
\end{equation}
Substitute~\eqref{eqn:exp-integral-1} and~\eqref{eqn:exp-integral-2} into the inequality~\eqref{eqn:exp-bd-integrals} to reach
$$
\trace \big[ \mtx{C} \big( \econst^{\mtx{A}} - \econst^{\mtx{B}} \big) \big]
	\leq \frac{1}{2} \bigg( \trace\big[ \mtx{C}^2 \big( \econst^{\mtx{A}} + \econst^{\mtx{B}} \big) \big] \cdot
	\trace\big[ (\mtx{A}-\mtx{B})^2 \big( \econst^{\mtx{A}} + \econst^{\mtx{B}} \big) \big] \bigg)^{1/2}.
$$
We obtain the result stated in Lemma~\ref{lem:emvti} by applying
the numerical inequality between the geometric mean and the
arithmetic mean.
\end{proof}

\bibliographystyle{plainnat}
\bibliography{refs-stein_kernel.bib}

\end{document}